\documentclass[a4paper]{amsart} 
\usepackage{amssymb}
\usepackage{dsfont}
\usepackage{hyperref}
\usepackage[a4paper]{geometry}

\theoremstyle{plain}
\newtheorem{thm}{Theorem}[section] 
\newtheorem{prop}[thm]{Proposition} 
\newtheorem{lem}[thm]{Lemma} 
\newtheorem{cor}[thm]{Corollary} 
\theoremstyle{remark} 
\newtheorem*{rem}{Remark}
\theoremstyle{definition}
 
\newtheorem*{conditionHD}{Condition (HD)}
\numberwithin{equation}{section}  

\renewcommand*{\div}{\operatorname{div}}
\newcommand*{\curl}{\operatorname{curl}}
\newcommand*{\id}{\operatorname{Id}}
\newcommand*{\supp}{\operatorname{supp}}
\newcommand*{\loc}{\mathrm{loc}}

\newcommand*{\bydef}{\overset{\rm def}{=}}
\newcommand*{\norm}[1]{\left\Vert #1\right\Vert}

\usepackage[textsize=tiny,textwidth=27mm]{todonotes}

\title[From QGSW to Euler Equations]{Sharp Strong Convergence in Ideal Flows}
	
\author{Haroune Houamed}
\address{New York University Abu Dhabi \\
		Abu Dhabi \\
		United Arab Emirates} 
\email{\href{mailto:haroune.houamed@nyu.edu}{haroune.houamed@nyu.edu}}

\author{Marc Maga\~na}
\address{Departament de Matem\`atiques – UAB \\
		Facultat de Ci\`encies, 08193 – Bellaterra, Barcelona, Spain} 
\email{\href{mailto:marc.magana@uab.cat}{marc.magana@uab.cat}}

\begin{document}
	
	\date{\today}
	
	\begin{abstract}
		We investigate the strong convergence  of weak solutions to the two-dimensional Quasi-Geostrophic Shallow-Water (QGSW) equation as the inverse Rossby radius tends to zero. In this limit, we recover the Yudovich solution of the incompressible Euler equations.
		  We prove that the vorticity convergence  holds  in   $L^\infty_t L^p_x$, for any finite integrability exponent $p<\infty$. This extends to the case $p=\infty$ provided that the initial vorticities are continuous and converge uniformly. We  also discuss the sharpness of this limit by demonstrating that the continuity assumption on the initial data is necessary for the endpoint convergence in $L^\infty_{t,x}$. 
		
		The proof of the strong convergence relies on the {\em Extrapolation Compactness} method, recently introduced by Ars\'enio and the first author to address similar stability questions for the Euler equations. The approach begins with establishing the convergence in a lower regularity space, at first. Then, in a later step, the convergence to Yudovich's vorticity of Euler equations in Lebesgue spaces comes as a consequence of a careful analysis of the evanescence of specific high Fourier modes of the QGSW vorticity. A central challenge arises  from the absence of a velocity formulation for QGSW, which we overcome by employing advanced  tools from Littlewood Paley theory in   endpoint settings.
		
		  The sharpness of the convergence in the endpoint $L^\infty_{t,x}$ case is obtained in the context of vortex patches, drawing insights from key  findings  on uniformly rotating and stationary solutions of active scalar equations. 
	\end{abstract}

	\maketitle
	
	\tableofcontents


	\section{Introduction and Main results}

	\subsection{Setup and Motivation}
	The quasi-geostrophic shallow-water \eqref{QGSW} equations serve as an essential model for describing large-scale atmospheric and oceanic circulations. These equations arise asymptotically from the rotating shallow-water equations in the limit of rapid rotation and weak variations of the free surface (for a detailed derivation an a comprehensive overview, see \cite{plotka2011shallow,vallis2017atmospheric} and references therein). The QGSW equations take the form  
	\begin{equation}\label{QGSW} \tag{QGSW}
		\begin{cases}
			\begin{aligned}
				\partial_t \omega_\lambda + u \cdot \nabla \omega_\lambda &=0,  \\
				u&= \nabla ^\perp \left( \lambda - \Delta\right)^ {-1} \omega_\lambda , \\
				\omega_\lambda|_{t=0} &= \omega_{0, \lambda}, 
			\end{aligned}
		\end{cases}
	\end{equation}
	where $\omega_\lambda$ is the potential vorticity, $u$ the velocity field, $\nabla^\perp= (-\partial_2, \partial_1)^t$, and $\lambda\geq 0$. In the literature, $\sqrt{\lambda}$ is known as the inverse Rossby radius, a characteristic length scale that emerges from the balance between rotation and stratification. A small $\lambda$ corresponds to a nearly rigid free surface, and in the limit $\lambda \to 0$, the QGSW equations reduce to the two-dimensional vorticity formulation of the Euler equations 
	\begin{equation}\label{Euler} \tag{Euler}
		\begin{cases}
			\begin{aligned}
				\partial_t \Omega + v \cdot \nabla \Omega &=0,  \\
				v&= \nabla ^\perp  (-\Delta)^{-1} \Omega , \\
				\Omega|_{t=0} &= \Omega_{0}. 
			\end{aligned}
		\end{cases}
	\end{equation} 
	
	The literature on the Euler equations is vast, and we highlight here only some key results with a direct connection to the scope of this paper. Global existence of classical solutions is a consequence of the transport structure of the vorticity equation, where the conservation of Lebesgue norms of the initial vorticity plays a central role in ruling out any singularity formation in dimension two. We refer to \cite{C98,MB01} for a more comprehensive discussion on relevant subject matters.
	
	 For less regular initial data, Yudovich \cite{Y63} proved in 1962  that Euler's system admits a unique global weak solution when the initial vorticity $\Omega_0$ lies in $L^1_x\cap L^\infty_x$. This provides a perfect context for the study of vortex patch solutions, namely   characteristic functions of bounded domains, including (rigid/non-rigid) uniformly rotating patches. A remarkable exact solution of such type  is Kirchhoff’s ellipse (1876) \cite{K1876}, established as a first   nontrivial example of a uniformly rotating solution. The existence of general class of rotating patches, called also V-states, was discovered numerically by Deem and Zabusky (1978) \cite{DZ78}. An analytical proof based on bifurcation theory and complex analysis was  provided later by Burbea (1982) \cite{B82}, establishing the existence of \( m \)-fold symmetric V-states bifurcating from Rankine vortices for any \( m \in \mathbb{N}^* \). More recent studies have explored V-states in various contexts and models, including doubly-connected patches, vortex pairs, and boundary effects; see \cite{CCGS16II,CCGS16,dHHH16,dHHMV16, dHHHM16,G20,G21,GPSY21,HH15,HH21,HMW20,HW22,HM16,HM16II,HM17,HMV13,HMV15} for more details. 
	
	Parallel developments in the \eqref{QGSW} equations began in the late 1980s, initially through numerical studies. Polvani (1988) \cite{P88} and Polvani, Zabusky, and Flierl (1989) \cite{PFZ89} investigated generalizations of Kirchhoff’s ellipse for different values of the parameter, including doubly-connected patches and multi-layer flows. More recently, P\l otka and Dritschel (2012) \cite{PD12} carried out a more detailed analysis, studying the linear stability and nonlinear evolution of these solutions. The first analytical result came much later when Dritschel, Hmidi, and Renault (2019) \cite{renault} extended Burbea's result to \eqref{QGSW}. Since then, the mathematical study of \eqref{QGSW} has gained momentum, as seen for instance in the recent works \cite{hmidi2021time,MMO25,roulley2023vortex,YAN20241}.
	
	Given this concise overview of the state of   the art concerning both \eqref{QGSW} and \eqref{Euler}, it becomes particularly compelling to rigorously understand the formal connection between these models as $\lambda$ approaches zero. While Mateu, Orobitg and the second author address  this problem in \cite{MMO25} within the context of H\"older continuous initial data, we are  particularly  interested here in analyzing  this connection within a rougher  configuration,  specifically one that allows for the consideration of more general and uniformly rotating vortex patches.
More precisely, we investigate the convergence of Yudovich-type solutions of the QGSW equations to those of the Euler equations as the inverse Rossby radius, $\sqrt{\lambda}$, tends to zero. Our main results establish rigorous, sharp convergence under suitable assumptions, providing quantitative insights into the relationship between these two models in a critical setting. A more detailed discussion of the contributions of this paper is provided in the next section.

	It is worth noting that the problem of interest here is closely related to the inviscid limit in the Navier--Stokes equations within the Yudovich class. For earlier results on classical solutions, see \cite{BM81, C86, K72, S71}, and for more recent developments on weak solutions, see \cite{AD04, CW95, CW96, NM07} for vortex patches and \cite{CH96CPDE} for the case of a non-Lipschitz velocity field. More recently, another significant progress has also been made on the strong vanishing viscosity limit of vorticities in Lebesgue spaces within the Yudovich class, as established in \cite{CCS, CDE, HCE}. A brief discussion of the relevance of that limit to our results in this paper will be provided in subsequent remarks, later on.

	\subsection{Main results: Statements and    Discussions}

	 The primary aim of this paper is to rigorously establish the limit as $\lambda \to 0$ in \eqref{QGSW}, thereby deriving the Euler equations from the Quasi-Geostrophic Shallow-Water model. Additionally, we seek a quantitative description of this limit in a borderline setting, encompassing key scenarios such as vortex patches where the velocity field is not necessarily Lipschitz-continuous. Furthermore, we aim to establish a sharp stability result in terms of the functional setting in which the solutions initially reside. This forms the core of our findings, which we state below (in Theorems \ref{thm:main} and \ref{thm:main2}).
	\begin{thm}[Strong Convergence in Yudovich class]\label{thm:main}
		 Let $\omega_{0,\lambda}$ and $\Omega_0$ be an initial data for \eqref{QGSW} and \eqref{Euler}, respectively, belonging to the Yudovich class, i.e., 
		\begin{equation*}
			\omega_{0,\lambda}, \Omega_0 \in L^1\cap L^\infty(\mathbb R^2),
		\end{equation*}
		uniformly in $\lambda>0$. Assume further that 
		\begin{equation*}
			\lim_{\lambda\to 0} \left(  \norm {\nabla^\perp \Delta^{-1}(  \omega_{0,\lambda}-\Omega_0)}_{ L^2(\mathbb R^2)}+  \norm {\omega_{0,\lambda}-\Omega_0}_{ L^2(\mathbb R^2)}\right) =0.
		\end{equation*}
		Then, the solution $\omega_\lambda$ of \eqref{QGSW} converges, as $\lambda\to 0$, to the solution $\Omega$ of  \eqref{Euler} on any finite interval of time. More precisely, it holds that 
		\begin{equation*}
			\lim_{\lambda\to 0} \left(  \norm {\nabla^\perp \Delta^{-1}(  \omega_{\lambda}-\Omega)}_{L^\infty ([0,t]; L^2\cap L^\infty(\mathbb R^2))}+  \norm {\omega_{\lambda}-\Omega}_{L^\infty ([0,t]; L^p(\mathbb R^2))}\right) =0,
		\end{equation*}
		for any $t\geq 0$, and $p\in (1,\infty)$. 
		
		If moreover $\Omega_0$ is continuous and 
		\begin{equation*}
			\lim_{\lambda\to 0}    \norm {\omega_{0,\lambda}-\Omega_0}_{ L^\infty (\mathbb R^2)}  =0,
		\end{equation*}
		then it holds that 
		\begin{equation*}
			\lim_{\lambda\to 0}   \norm {\omega_{\lambda}-\Omega}_{L^\infty ([0,t]; L^\infty (\mathbb R^2))}  =0,
		\end{equation*}
		for any $t>0$.
	\end{thm}

	Given the preceding theorem, a natural question that arises is about the sharpness of the additional continuity assumption on the initial data, which ensures the convergence in   $L^\infty_{t,x}$. In the following theorem, we confirm that this extra assumption is indeed necessary, thereby establishing the sharpness of the conditions under which Theorem \ref{thm:main} holds.

	\begin{thm}[Sharpness of the Endpoint Convergence]\label{thm:main2} 
		 There is a class of discontinuous functions such that, for any $\Omega_0 $ belonging to such a class, the associated solutions $\omega_\lambda$ and $\Omega$ of \eqref{QGSW} and \eqref{Euler}, respectively,  with the same initial data $\omega_{0,\lambda}=\Omega_0$, enjoy the lower bound 
		 \begin{equation*}
		 	\norm {\omega_{\lambda}-\Omega}_{L^\infty ((t_1,t_2);L^\infty(\mathbb R^2))} \geq 1,
		 \end{equation*}
		 for any $t_2> t_1>0$ and $\lambda>0.$
	\end{thm}

	Before diving into the technical discussion of the challenges and our method of proof, let us first highlight some key observations regarding the preceding statements.
		 
		 \begin{rem}[Yudovich solutions] Throughout the paper, we focus our analysis on Yudovich-type solution $\omega_\lambda$ to \eqref{QGSW}. It can be shown that for any initial data $\omega_{0,\lambda}$,  as specified above, there is a unique global solution of \eqref{QGSW} enjoying the bound
		 \begin{equation*}
		 	\omega_\lambda \in L^\infty(\mathbb R^+;L^1\cap L^\infty(\mathbb R^2)), \quad \text{uniformly in } \lambda \in (0,1) .
		 \end{equation*} 
		 Moreover, this solution is continuous in time with values in $L^p(\mathbb R^2)$ for any finite exponent $p$. The uniqueness aspect of this statement can be derived using the stability arguments presented later on during the proof of Theorem \ref{thm:main}. Another approach to proving uniqueness in the Yudovich class involves examining the convergence of the flow maps associated with two distinct solutions, as done in \cite{GG24} for general   active scalar equations. For an adaptation of the techniques from the latter work to the quasi-geostrophic shallow-water equation, we refer to \cite{MMO25}. We also refer to \cite{HZHH23} for an application of similar techniques   to a system of active scalar equations that can be viewed as a combination of \eqref{QGSW} and \eqref{Euler}.
		 \end{rem}
		 
	\begin{rem}[The torus case]
		It is to be emphasized through our discussion below that the statement of Theorem \ref{thm:main} holds in the two-dimensional torus as well, without requiring any substantial amendments to our proof. This highlights the robustness of our method, which is based on Fourier analysis---a point that we will elaborate on further in the discussion below	.
		\end{rem}
	
	\begin{rem}[Rate of convergence] According to our proof below, the convergence of the velocities, with the same initial data $\omega_{0,\lambda}=\Omega_0$, occurs with a rate that follows from the bound 
	\begin{equation*}
		 \norm {\nabla^\perp \Delta^{-1}(  \omega_{\lambda}-\Omega)}_{L^\infty ([0,T]; L^2 (\mathbb R^2))} \lesssim _T \left(   C_{*} \lambda T \right)  ^{\frac{1}{2} \exp (-C_{*}T)} \log ^{\frac{1}{2}}\left( \frac{1}{  \lambda T }\right) ,
	\end{equation*}
		for sufficiently small $\lambda>0$, where    $C_{*}>0$ is a constant depending  on Lebesuge norms of the initial data. A more precise statement can be found in Theorem \ref{thm:conv:low:reg}, below. Notably, the   algebraic  component of this convergence rate is identical to  that in the vanishing viscosity limit of the Navier--Stokes solution in the Yudovich class. 
				
		In the case of \eqref{QGSW}, the additional logarithmic error in the above quantitative estimate stems from our inability to directly establish a stability result in the $L^2_x$-space for the velocity field. More specifically, this error arises due to proving the convergence of velocities, at first, in a negative Besov space that scales as $B^{-1}_{2,\infty}$,  combined with a functional inequality of the form:
		\begin{equation*}
			\norm {f}_{L^2(\mathbb R^2)} \lesssim \norm f_{B^{-1}_{2,\infty}} \log ^\frac{1}{2}\left(\frac{\norm { f}_{ H^1}}{\norm {f}_{B^{-1}_{2,\infty}}}\right).
		\end{equation*}
		A plausible interpretation of this additional error, compared to the vanishing viscosity limit, is the absence of a velocity formulation for \eqref{QGSW}, unlike in the Navier--Stokes and Euler equations.

		On the other hand, we do not expect any explicit rate of convergence for the vorticities in Lebesgue spaces without additional regularity of the initial data. For example, in the case of vortex patches, this specific type of data possesses extra Besov-type regularity, namely $B^{s}_{2,\infty}(\mathbb R^2)$, for some $s>0$ (see \cite[Lemma 3.2]{CW96}). In this case, a straightforward interpolation inequality
		\begin{equation*}
			B^{-1}_{2,\infty} \cap B^s_{2,\infty} \hookrightarrow L^2(\mathbb R^2)
		\end{equation*}
		 yields the convergence rate
        \begin{equation*}
			\norm {\omega_{\lambda}-\Omega}_{L^\infty ([0,t]; L^2(\mathbb R^2))} =_t \mathcal O \left( \lambda ^{\frac{s \exp (-2C_{*} T)}{2(1+s \exp (-C_{*} T))} } \right),
		\end{equation*}
which is  comparable to the findings from \cite{CDE} on the vanishing viscosity limit in Yudovich class.
 
On top of that, it is worth emphasizing that this rate can be slightly improved using the method of proof we present below, which is based on the {\em  extrapolation compactness} technique introduced by Ars\'enio and the first author in \cite{ah2}. In line with \cite[Theorem 2.1]{ah2} concerning the rate of convergence of vanishing viscosity of Yudovich solutions to the Navier--Stokes equation, the following improved rate of convergence can be derived 
        \begin{equation*}
			\norm {\omega_{\lambda}-\Omega}_{L^\infty ([0,T]; L^2(\mathbb R^2))} =_T \mathcal O \left(  \frac{\lambda ^{ s \exp (-2C_{*} T)  }}{|\log \lambda|} \right)^{\frac{1}{2(1+s \exp (-C_{*} T))}}, \quad \text{as } \lambda \to 0,
		\end{equation*}
by  building on the quantitative estimate of the very high frequencies of solutions to general transport equations from Proposition \ref{prop:high:transport}, below. We choose to not address the full details of this matter in the present work.

	\end{rem}

\begin{rem}[Sharpness in the endpoint case] It is to be emphasized later on that  Theorem \ref{thm:main2} holds for any initial vorticity of the form 
\begin{equation*}
	\omega_0=\mathds 1_{D},
\end{equation*} 
where $D\subset \mathbb R^2$ is any nonempty simply connected domain (or a finite union of simply connected domains), with a rectifiable boundary, that is different than the disc. A further discussion about the non-convergence in the $L^\infty_{t,x}$ space in this case is deferred to the subsequent paragraph.
\end{rem}

	\subsection{Challenges and  Methods of proof}\label{section:challenge_and_proof}
	
	For clarity, we briefly outline here the main challenges encountered in proving Theorems \ref{thm:main} and \ref{thm:main2}, as well as the key elements of our approach. The proof of Theorem \ref{thm:main} will be provided in Section \ref{section:proof:thm1}, formulated as a consequence of a series of intermediate results developed in Section \ref{section:three:levels}. On the other hand,  the proof of Theorem \ref{thm:main2} will be discussed in Section \ref{section:proof:thm2}, after establishing the necessary prerequisites used in its demonstration. 
	
	For convenience, we divide this discussion into two parts, each offering concise insights into the methods of proof  for Theorems \ref{thm:main} and \ref{thm:main2}, respectively.
	
	\subsubsection*{Strong Limit in the small inverse Rossby regime} 
	
	The overall approach to proving Theorem \ref{thm:main} draws on insights from the procedure  introduced in \cite{ah2}, which has been used to study both the inviscid limit from the Navier--Stokes to Euler equations and the non-relativistic limit from the Euler--Maxwell system to an inviscid MHD system, all within the Yudovich class of solutions.

The proof roadmap is structured in three stages, each detailed in the Subsections \ref{section:level:one}, \ref{section:level:two} and \ref{section:level:three}, respectively. The strategy itself is grounded in a simple abstract observation, known as {\em Extrapolation Compactness}, first introduced in \cite{ah2}; see Lemma 1.5 therein.

	In \underline{\textbf{Level One}}, we establish a convergence result in a lower regularity space. A key challenge here, compared to the inviscid limit problem, is the absence of a velocity reformulation for \eqref{QGSW}. Despite this, we manage to show the convergence of the vorticities in a Besov-type space with negative regularity, which is strictly embedded between $H^{-1}$ and $B^{-1}_{2,\infty}$. Roughly speaking, in terms of scaling, this is comparable to the classical vanishing viscosity limit, see \cite{C98, CW96, MB01} for instance.

Achieving this strong convergence relies on a refined combination of several Fourier and Littlewood--Paley techniques, including commutator-type estimates in endpoint settings. It is important to emphasize that the convergence at this level is accompanied by a rate that is, once again, comparable to the vanishing viscosity limit in the context of Navier--Stokes equations, in terms of scaling. This ``sharp'' rate arises from a detailed analysis of the velocity-vorticity relationship in the solution to \eqref{QGSW}. For more comprehensive details, see Section \ref{section:level:one}.
	
	In \underline{\textbf{Level Two}}, building on the preceding convergence result, we observe that the convergence in a higher regularity space---specifically, the convergence in the Lebesgue space $L^\infty_tL^2_x$---is equivalent to the vanishing of certain very high frequencies in $L^\infty_tL^2_x$ for the solution of \eqref{QGSW}. This is the content of Proposition \ref{proposition:reduction}.

	In \underline{\textbf{Level Three}}, we begin with addressing a slightly more general problem. More precisely, we examine  the behavior of the high frequencies of a general transport equation:
	\begin{equation*} 
		\partial_t g + h\cdot \nabla g = F, \qquad g|_{t=0}=g_0,
	\end{equation*}
	where $F$ is a source term (potentially dependent on  $g$), and $h$ is a divergence-free vector field. One could think of the velocity field  $h$ as potentially being recovered from the solution $g$ via some form of Biot--Savart law.

	Assuming   that   $h$, $g_0$ and/or $F$, and eventually $g$, depend on some parameter $\delta\in (0,1)$ in such a way that the dependence of the   problem's inputs (i.e., $F$, $g_0$, and possibly $h$ in a rough sense) is compact in Lebesgue spaces,  we then investigate whether this compactness property can be extended to $g(t)$, for any positive times $t>0$.
	
	In Proposition \ref{prop:high:transport}, we provide a quantitative bound on the behavior of the very high frequencies of   $g$, which is crucial for deriving a complete statement on the compactness of $g(t)$,  for any positive times $t>0$, under reasonable assumptions on the velocity field $h$, the source term $F$ and the initial datum $g_0.$
	
	The takeaway from Proposition \ref{prop:high:transport} on the more general problem described above enables us to deduce that the family 
		of solutions $\left( \mathds 1_{|D|\geq \Theta_\lambda} \omega_\lambda  \right)_{\lambda \ll1}$  to \eqref{QGSW}
		is strongly compact in $L^\infty_\loc  (\mathbb R^+ ;L^2(\mathbb R^2))$, converging to zero in $L^\infty  ([0,T];L^2(\mathbb R^2))$, for any finite time $T>0$ and any family of parameters $(\Theta_\lambda)_{\lambda\in (0,1)}$ satisfying 
		\begin{equation*}
			\lim _{\lambda\to 0} \Theta_\lambda = \infty \quad \text{ and } \quad \lim_{\lambda \to 0} \Theta_\lambda \left(  \norm {\nabla^\perp \Delta^{-1}(  \omega_{0,\lambda}-\Omega_0)}_{ L^2(\mathbb R^2)} ^2 + C_{**} \lambda  T \right)  ^{\frac{1}{2} \exp (-C_{**}T)}=0,
		\end{equation*}
where $C_{**}>0$ is a constant depending on the size of the initial data uniformly for $\lambda \ll1$.

Given this, the convergence 
\begin{equation*}
			\lim_{\lambda\to 0}   \norm {\omega_{\lambda}-\Omega}_{L^\infty ([0,t]; L^p(\mathbb R^2))}  =0,
		\end{equation*}
		for $p=2$ is obtained first,  by employing Proposition \ref{proposition:reduction}, while the case $p\in (1,\infty)$ follows thereafter by interpolation.
		
 Establishing the convergence in the endpoint case $p=\infty$, under the additional assumption that $\Omega_0$ is continuous, comes as a consequence of the convergence of the flow map of \eqref{QGSW} to the one of \eqref{Euler} in strong topologies. Here, a convergence of the flow maps  in $L^\infty_{t,x}$ is sufficient, although   not optimal. This convergence, in turn, follows from the strong convergence  of the velocities in suitable functional spaces.

\subsubsection*{Sharpness of the Endpoint Convergence} 

The proof of Theorem \ref{thm:main2} is based on a result of independent interest. A relevant question that can be posed in this direction is the following: {\em  Given an initial datum $\Omega_0=\omega_{0,\lambda}$, how does the solution $\Omega(t)$ of \eqref{QGSW} differs from the solution $\omega_{\lambda}(t)$ of \eqref{Euler} at positive times $t>0$, for  $\lambda>0$?} 

Throughout the proof of Theorem \ref{thm:main2}, we address this problem in the context of vortex patches. We show that, in this case, if the solutions $\Omega$ and $\omega_{\lambda}$ were to be identical on a given arbitrary  interval of time $(t_1,t_2)$, then we are led to the question of   characterizing   solutions of the equation 
  \begin{equation}\label{EQ:boudary}
				\nabla ^\perp \left( \mathcal K_\lambda*  \mathds 1_{\Phi(t,D)} \right) \cdot \vec n_{t} =0, \qquad \text{on } \partial \Phi(t,D), \qquad \text{for all } t\in (t_1,t_2),
			\end{equation}
			where $D$ is the domain associated with the initial patch,  $\Phi (t,D)$ is the transported domain  by the flow map associated with the solution $\Omega=\omega_{\lambda}$, while $\vec n_t$ refers to the outward normal vector to the boundary $ \partial \Phi (t,D)$, and     $\mathcal K_\lambda$  is the kernel associated with the operator 
			\begin{equation*}
				  \Delta^{-1} +  (\lambda-\Delta)^{-1}  .
			\end{equation*}
We then draw insight from a known result from \cite{GPSY21} showing that, if the kernel $\mathcal K_\lambda$ and the domain $D$ satisfy some reasonable conditions---proven to hold in our case in the discussion at the beginning of Section \ref{section:proof:thm2}---then the only solution of \eqref{EQ:boudary} is  (a translated) disc. This rules out the possibility of having identical solutions, if one assumes  initially that the  patch at time $t=0$ is different than the stationary solution associated with a disc.

Finally, obtaining the lower bound stated in Theorem \ref{thm:main2} follows by estimating  the difference between two distinct patches, which is the only remaining scenario based on the preceding analysis.

	\section{Tool Box} 
	 
	 In this section, we recall several fundamental tools that will be utilized in the proofs of our main theorem. We also establish some key lemmas that will come in handy in the subsequent sections, later on.
	 
	 Before proceeding any further, let us fix some notations that will be used throughout the paper. The inequality $A\lesssim B$ refers to $A \leq CB$ for some constant $C>0$ that is independent of $A$ and $B$. The symbol $[P,Q]$ denotes the commutator of two operators $P$ and $Q$, and it is given by 
	 \begin{equation*}
	 	[P,Q]= PQ - QP.
	 \end{equation*}
	 Finally, for a given smooth function $\phi$, the operator $\phi (D)$ denotes the Fourier multiplier by the symbol $\phi (\xi)$, where $\xi$ refers to the Fourier variable. It can also be interpreted as the convolution operator with the inverse Fourier transform of $\phi$.
	
\subsection{Littlewood-Paley theory and Besov spaces}\label{section:littlewood_paley}
 Throughout the paper, we   rely on a number of classical results concerning Littlewood--Paley theory and  Besov spaces. The building blocks of all of that is the dyadic partitions of unity 
\begin{equation*}
	1= \psi(\xi)+\sum_{k=0}^\infty \varphi\left(2^{-k}\xi\right),
	\quad\text{for all }\xi\in\mathbb{R}^d,
\end{equation*}
where the functions in the preceding identity are $C^\infty_c $-smooth and enjoy the properties that 
\begin{equation*}
	\psi,\varphi\geq 0\text{ are radial},
	\quad\supp\psi\subset\left\{|\xi|\leq \frac{4}{3}\right\},
	\quad\supp\varphi\subset\left\{\frac{3}{4}\leq |\xi|\leq \frac{8}{3}\right\},
\end{equation*}
and 
\begin{equation*}
	\psi|_{B(0,1)}=1.
\end{equation*}

Given that, we introduce the Fourier-multiplier operators, defined on the space tempered distributions $\mathcal S'(\mathbb R^d)$ 
\begin{equation*}
	S_k,\,\Delta_k:
	\mathcal{S}'\left(\mathbb{R}^d\right)\rightarrow\mathcal{S}'\left(\mathbb{R}^d\right),
\end{equation*}
with $k\in\mathbb{Z}$, as 
\begin{equation*} 
	S_k f\bydef \psi_k(D) f
	= \left(\mathcal{F}^{-1}\psi_k\right)*f
	\quad\text{and}\quad
	\Delta_k f\bydef \varphi_k(D)f
	= \left(\mathcal{F}^{-1}\varphi_k\right)*f,
\end{equation*}
where the rescaled cutoffs are defined, in the inhomogeneous context, as 
\begin{equation*}
	\psi _{k}(\xi)\bydef\psi\left(2^{-k}\xi\right),\text{ for } k\geq 0, \qquad \text{and}\quad  \psi_{k}(\xi)\bydef 0, \quad    \text{ for } k\leq -1,
\end{equation*}
while 
\begin{equation*}
	\varphi_k (\xi)\bydef\psi\left(2^{-k}\xi\right), 
	 \text{ for } k\geq 0, \qquad \text{and}\quad  \varphi_k (\xi)\bydef 0, \quad    \text{ for } k\leq -2,
\end{equation*}
with   
\begin{equation*}
	\varphi _{-1} (\xi) \bydef \psi (\xi).
\end{equation*}
Note in passing that one can make such a choice of cutoff functions such that 
\begin{equation*}
	S_k = \sum_{j\leq k-1} \Delta_j, \quad \text{in } \mathcal S'(\mathbb R^d).
\end{equation*}
The Littlewood--Paley decomposition of $f\in\mathcal{S}'$ is then given by
\begin{equation*}
	f= S_0 f+\sum_{k=0}^\infty\Delta_{ k}f,
\end{equation*}
where the series has a sense in $\mathcal{S}'(\mathbb R^d)$.

For any $s \in \mathbb{R}$ and $1\leq p,q\leq \infty$, the inhomogeneous Besov space $ B^{s}_{p,q}\left(\mathbb{R}^d\right)$ is defined  as the subspace of tempered distributions   endowed with the norm
\begin{equation*}
	\left\|f\right\|_{ B^{s}_{p,q}\left(\mathbb{R}^d\right)}=
	\left(
	\sum_{k\in\mathbb{Z}} 2^{ksq}
	\left\|\Delta_{k}f\right\|_{L^p\left(\mathbb{R}^d\right)}^q\right)^\frac{1}{q},
\end{equation*}
if $q<\infty$, and
\begin{equation*}
	\left\|f\right\|_{ B^{s}_{p,q}\left(\mathbb{R}^d\right)}=
	\sup_{k\in\mathbb{Z}}\left(2^{ks}
	\left\|\Delta_{k}f\right\|_{L^p\left(\mathbb{R}^d\right)}\right),
\end{equation*}
if $q=\infty$. This space is a Banach space, for any value of parameters $s$, $p$ and $q$.  

For a more thorough discussion on the functional properties of Besov spaces, we refer to \cite{bcd11}.

\subsection{Technical lemmas} In this section, we compile and establish several key results that not only serve our analysis but are also of independent interest.

We begin with some commutator estimates which have long been a powerful tool in the analysis of incompressible ideal flows. In the configuration relevant to our study, the adapted version below will play a crucial role in establishing the convergence of vorticities in spaces of lower regularity.

While the general statement of this estimate is based on classical results (see \cite[Section 2.10]{bcd11}), we were unable to locate an identical version in the existing literature that directly applies to our context. Therefore, for the sake of clarity and completeness, we have chosen to provide a full, self-contained proof of the following lemma.
	\begin{lem}[Commutator estimate]\label{lemma:commutator} let $v$ be a divergence-free vector field. In any dimension $d\geq 2,$  	It holds that 
	\begin{equation*}
		[\Delta_j,v\cdot \nabla ] f = \mathcal I_j + \mathcal J_j,
	\end{equation*}
	where 
	\begin{equation*}
		\sup_{j \in \mathbb Z}\left( 2^{-j}\norm {\mathcal I_j}_{L^{p_1}(\mathbb R^d)} \right)\lesssim \norm {\nabla v}_{L^{m_1}(\mathbb R^d)} \norm f_{B^{-1 }_{{q_1},\infty } (\mathbb R^d)}   ,
	\end{equation*}
	and 
	\begin{equation*}
		\sup_{j \in \mathbb Z}\left( 2^{-j}\norm {\mathcal J_j}_{L^{p_2}(\mathbb R^d)} \right)\lesssim \norm {\nabla v}_{B^{0}_{m_2,\infty} (\mathbb R^d)} \norm f_{B^{-1}_{q_2,1}(\mathbb R^d)}   ,
	\end{equation*}
	for any   set of parameters in $[1,\infty]$ satisfying   
	\begin{equation*}
		\frac{1}{m_1} + \frac{1}{q_1}= \frac{1}{p_1}, \qquad \frac{1}{m_2} + \frac{1}{q_2}= \frac{1}{p_2}.
	\end{equation*}
	\end{lem}
	
	\begin{proof} For simplicity, we are going to present the whole proof with the notation 
\begin{equation*}
	(m,p,q)=(m_1,p_1,q_1)=(m_2,p_2,q_2).
\end{equation*}
		We shall establish a slightly stronger estimate. More precisely, in terms of Bony's paraproduct decomposition (see \cite[Section 2.8.1]{bcd11}),  we consider the decomposition    
		\begin{equation*}
			  [\Delta_j,v\cdot \nabla ] f = \mathcal I_j + \mathcal J_j + \mathcal K_j,
		\end{equation*}
		where 
		\begin{equation*} 
				\mathcal I_{j}= [T_{v},\Delta_j] \nabla  f, 
			\qquad 
				\mathcal J_{j}= - \Delta_j \div   R( f  ,  v),
		\end{equation*}
and
        \begin{equation*}
			\begin{aligned}
				 \mathcal K_j  = T_{ \nabla  \Delta_j  f}  v  - \Delta_j T_{\nabla f}v +  \div   R( \Delta_j f  ,  v),
			\end{aligned}
		\end{equation*}
		and we claim that $\mathcal I_j$ and $\mathcal J_j$ satisfy the result of the lemma, whereas $\mathcal K_j$ enjoys the stronger bound 
		\begin{equation*}
		\sup_{j \in \mathbb Z}\left( 2^{-j}\norm {\mathcal K_j}_{L^{p}} \right)\lesssim \norm {\nabla v}_{B^{0}_{m,\infty} } \norm f_{B^{-1}_{q,\infty}}   .
	\end{equation*} 
		For the first term, we first notice that  
		\begin{equation*}
			\mathcal I_{j} = \sum_{|j'-j|\leq 4} [\Delta_j, S_{j'-1} v] \nabla \Delta_{j'} f,
		\end{equation*}
		 which yields, by applying the standard commutator estimate \cite[Lemma 2.97]{bcd11}, that 
		 \begin{equation*}
		 	\norm {\mathcal I_{j}}\lesssim \norm {\nabla v}_{L^m} \sum_{|j-j'|\leq 4}  2^{j-j'}  \norm {\Delta_{j'} f}_{L^q}.
		 \end{equation*}
		Therefore, we deduce that 
		\begin{equation*}
		\sup_{j \in \mathbb Z}\left( 	2^{-j} \norm {\mathcal I_{j}}\right) \lesssim \norm {\nabla v}_{L^m} \sup_{j \in \mathbb Z}\left( 2^{-j}\norm {\Delta_j f}_{L^q}  \right).
		\end{equation*}
		
		The second term can be estimated by employing the classical product laws  for the remainder of Bony's decomposition (see \cite[Theorem 2.85]{bcd11}). Accordingly, we obtain that  
		\begin{equation*}
			 \begin{aligned}
			 	\sup_{j \in \mathbb Z}\left( 2^{-j}\norm {\mathcal J_j}_{L^{p}} \right)
			 	& = \norm {\div R(v,f)}_{B^{-1}_{p,\infty}}
			 	\\
			 	&\lesssim  \norm {v}_{B^{1}_{m,\infty}} \norm {f}_{B^{-1}_{q,1}}.
			 \end{aligned}
		\end{equation*}
		
		As for the last term, let us split it into a sum of the terms $\mathcal K_j^1$, $\mathcal K_j^2$ and $\mathcal K_j^3$ which correspond to each summand in its definition, respectively. 
		 Then, for the first term in this new decomposition, we first notice that the identity 
		\begin{equation*}
			\mathcal K_{j}^1 = \sum_{|j'-j|\leq 4}   S_{j'-1} \nabla f  \Delta_{j'} v
		\end{equation*}
		yields the bound 
		\begin{equation*}
			\begin{aligned}
				\norm {\mathcal K_{j}^1  }_{L^p} 
				&\lesssim \norm {\nabla v}_{B^{0}_{m,\infty}} \sum_{|j-j'|\leq 4} 2^{j'-j} \norm {S_{j'-1} f}_{L^q}
				\\
				&\lesssim \norm {\nabla v}_{B^{0}_{m,\infty}} \sum_{|j-j'|\leq 4}  \sum_{k\leq j'-2}  \norm {\Delta_k f}_{L^q}
				\\
				&\lesssim \norm {\nabla v}_{B^{0}_{m,\infty}} \sum_{|j-j'|\leq 4}  \sum_{k\leq j'-2}  2^k \sup_{k\in \mathbb Z} \left( 2^{-k}  \norm {\Delta_k f}_{L^q} \right).
			\end{aligned}
		\end{equation*}
		Thus, we infer that 
		\begin{equation*}
			2^{-j} \norm {\mathcal K_{j}^1}\lesssim \norm {\nabla v}_{B^{0}_{m,\infty}} \sup_{j \in \mathbb Z}\left( 2^{-j}\norm {\Delta_j f}_{L^q}  \right).
		\end{equation*}
		
		As for the estimate of the second term $\mathcal K_{j}^2$, it follows by virtue of similar arguments  (see \cite[Theorem 2.82]{bcd11}, for instance) that lead to the control 
		\begin{equation*}
			\begin{aligned}
				\sup_{j \in \mathbb Z}\left( 2^{-j} \norm {\mathcal K_{j}^2}_{L^p}   \right) & =\norm {T_{\nabla f}v}_{B^{-1}_{p,\infty}}  
				\\
				&\lesssim \norm {\nabla f}_{B^{-2}_{q,\infty}}\norm {  v}_{B^{1}_{m,\infty}} 
				\\
				&\lesssim \norm {  f}_{B^{-1}_{q,\infty}}\norm {\nabla v}_{ B^0_{m,\infty}} .
			\end{aligned}
		\end{equation*}
		
		Finally, for the last term $\mathcal K_{j}^3$, we proceed   by first writing  that   		\begin{equation*}
			\begin{aligned}
				\sup_{j \in \mathbb Z}\left(   2^{-j} \norm {\mathcal K_{j}^3}_{L^p} \right) & =\sup_{j \in \mathbb Z}\left( 2^{-j} \norm {\div   R( \Delta_j f  ,  v)}_{L^p} \right)  
				\\
				&\lesssim  \sup_{j \in \mathbb Z}\left( 2^{-j} \norm { R( \Delta_j f  ,  v)}_{ \dot B^{1}_{p,1}} \right)  
				\\
				&\lesssim  \sup_{j \in \mathbb Z}\left( 2^{-j} \norm { \Delta_j f  }_{ \dot B^{0}_{q,1}} \right) \norm {\nabla v}_{ B^0_{m,\infty}}  .
			\end{aligned}
		\end{equation*}
		Therefore, employing the simple observation that  
		\begin{equation*}
			2^{-j}\norm {\Delta _j f}_{B^{0}_{p,1}} =   2^{-j }\sum_{|j-j'| \leq 1} \norm {\Delta_{j} \Delta_{j'} f}_{L^p}\leq   2^{-j }\sum_{|j-j'| \leq 1} \norm {\Delta_{j} f}_{L^p} \lesssim \sup_{j \in \mathbb Z}\left( 2^{-j} \norm { \Delta_j f  }_{ L^q}  \right)
		\end{equation*}
		yields the desired bound 
		\begin{equation*}
			\sup_{j \in \mathbb Z}\left(   2^{-j} \norm {\mathcal K_{j}^3}_{L^p} \right)  \lesssim  \sup_{j \in \mathbb Z}\left( 2^{-j} \norm { \Delta_j f  }_{ L^q}  \right)\norm {\nabla v}_{B^0_{m,\infty}}.  
		\end{equation*}
		This concludes the proof of the lemma.
	\end{proof}

	The next lemma that we prove below establishes a quantitative bound on the behavior of a quantity   that naturally appears when writing the difference between the quasi-geostrophic shallow-water and Euler's velocity fields. This is crucial in the proofs of our main results.
	
	\begin{lem}\label{lemma:error}
		For any $\omega \in L^1(\mathbb R^2)$, it holds that 
		\begin{equation*}
			\norm { \lambda \nabla^\perp(-\Delta)^{-1} (\lambda - \Delta)^{-1}\omega }_{\dot H^1\cap L^\infty (\mathbb R^2)} \lesssim \lambda^\frac{1}{2}   \norm {\omega}_{L^1(\mathbb R^2)}  ,
		\end{equation*}
		for any $\lambda >0$.
	\end{lem}
	\begin{proof}
		We introduce the function 
		\begin{equation*}
			\xi \mapsto S _\lambda (\xi) \bydef \frac{\lambda\xi}{|\xi|^2 (\lambda +|\xi|^2)} ,
		\end{equation*}
		which is clearly in $L^1(\mathbb R^2)$, for any fixed $\lambda>0$. Thus, its Fourier transform is well defined as a function in $C_{b}(\mathbb R^2).$  Accordingly,  we denote by $K_\lambda$ the convolution kernel associated with the Fourier symbol $S_\lambda(\xi)$, i.e.,  
	\begin{equation*}
		(\mathcal FK_ \lambda) (\xi) \bydef   S_\lambda(\xi), \quad \text{for all } \xi \in \mathbb R^2\setminus \{0\} .
	\end{equation*}
	Moreover, it holds that 
	\begin{equation*}
		\norm {K_\lambda}_{L^\infty(\mathbb R^2)} \lesssim \norm {S_\lambda}_{L^1(\mathbb R^2)} = \pi^2 \lambda^\frac{1}{2}.
	\end{equation*}
	Hence, the desired bound in $L^\infty$ follows directly from the estimate
	\begin{equation*}
		\begin{aligned}
			\norm { \lambda \nabla^\perp(-\Delta)^{-1} (\lambda - \Delta)^{-1}\omega }_{  L^\infty (\mathbb R^2)} 
			&= \norm { K_\lambda *\omega }_{  L^\infty (\mathbb R^2)}
			\\
			& \lesssim \lambda^\frac{1}{2}   \norm {\omega}_{L^1(\mathbb R^2)} .
		\end{aligned}
	\end{equation*}
	As for the $\dot H^1$ norm claimed in the lemma, it follows by instead applying the alternative bound on the kernel 
	\begin{equation*}
		\norm {\nabla K_\lambda}_{L^2(\mathbb R^2)} \lesssim \lambda \left(  \int_{\mathbb R^2} \frac{d\xi}{\left( \lambda + |\xi|^2\right)^2}  \right)^\frac{1}{2}
		 = \pi^{\frac{1}{2}}  \lambda^\frac{1}{2},
	\end{equation*} 
	which implies that 
	\begin{equation*}
		\begin{aligned}
			\norm { \lambda \nabla^\perp(-\Delta)^{-1} (\lambda - \Delta)^{-1}\omega }_{ \dot H^1(\mathbb R^2)} 
			&= \norm { \nabla  K_\lambda *\omega }_{  L^2 (\mathbb R^2)}
			\\
			& \lesssim     \norm {\nabla K_\lambda}_{L^2(\mathbb R^2)} \norm {\omega}_{L^1(\mathbb R^2)} 
			\\
			& \lesssim \lambda^\frac{1}{2}   \norm {\omega}_{L^1(\mathbb R^2)} .
		\end{aligned}
	\end{equation*}
	This completes the proof of the lemma.
	\end{proof}

	We conclude the  discussion of this section by stating a fundamental result in the context of interpolation inequalities in Besov spaces.  
	\begin{lem}[Interpolation] \label{lemma:interIneq} Let the dimension $d\geq 1$. For any  $s\in \mathbb R$, $\varepsilon>0$, and $p\in [1,\infty]$, it holds that 		\begin{equation*}
			\norm {    (\id-  S_0)  f   }_{B^{s}_{p,1} } \lesssim  \norm {    (\id-  S_0)  f   }_{B^{s}_{p,\infty} } \log \left( e+ \frac{ \norm {   (\id-  S_0) f  }_{B^{s+\varepsilon}_{p,\infty}} }{ \norm {    (\id-  S_0)  f   }_{B^{s}_{p,\infty} }}\right),
		\end{equation*} 
		for any  $f\in B^{s+\varepsilon}_{p,\infty}(\mathbb R^d)$.
	\end{lem}
	\begin{proof} 
		We write, for   $N\geq 1$ to be chosen at the end, that 
		\begin{align*}
			\norm {    (\id-  S_0)  f   }_{B^{s}_{p,1} } &= \norm {  S_0(\id-  S_0)  f   }_{L^p} + \sum_{0\leq j < N} 2^{js}\norm {    \Delta_j (\id-  S_0)  f   }_{L^p}+ \sum_{j \geq N} 2^{js}\norm {    \Delta_j (\id-  S_0)  f   }_{L^p}
			\\ &\leq \norm {  S_0(\id-  S_0)  f   }_{L^p} + N\sup_{j\geq 0} \left( 2^{js}\norm {    \Delta_j (\id-  S_0)  f   }_{L^p} \right) 
			\\
			& \quad + \sum_{j \geq N} 2^{-j\varepsilon} 2^{j(s+\varepsilon)}\norm {    \Delta_j (\id-  S_0)  f   }_{L^p} 
			\\ &\leq N 	\norm {    (\id-  S_0)  f   }_{B^{s}_{p,\infty} } + 2^{-N \varepsilon} \norm{ (\id - S_0) f }_ {\dot{B}^{s+\varepsilon}_{p,\infty}}.
		\end{align*}
		Optimizing on $N$ suggests the choice
		$$N=\frac{1}{\varepsilon}\log_2\left(\varepsilon \log(2) \frac{\norm{ (\id - S_0) f }_ {\dot{B}^{s+\varepsilon}_{p,\infty}}}{	\norm { (\id-  S_0)  f   }_{B^{s}_{p,\infty} }}\right)$$
		whence
		\begin{align*}
			\norm {    (\id-  S_0)  f   }_{B^{s}_{p,1} } &\lesssim \norm {    (\id-  S_0)  f   }_{B^{s}_{p,1} } \left(1+ \log\left( \frac{\norm{ (\id - S_0) f }_ {\dot{B}^{s+\varepsilon}_{p,\infty}}}{	\norm { (\id-  S_0)  f   }_{B^{s}_{p,\infty} }}\right)\right)\\& \lesssim  \norm {    (\id-  S_0)  f   }_{B^{s}_{p,\infty} } \log \left( e+ \frac{ \norm {   (\id- S_0) f  }_{\dot{B}^{s+\varepsilon}_{p,\infty}} }{ \norm {    (\id-  S_0)  f   }_{B^{s}_{p,\infty} }}\right).		
		\end{align*}
		Since the operator \(\id - S_0\) removes the low frequencies, the preceding bound holds for both the homogeneous and inhomogeneous Besov norm, thereby concluding the proof. 
	\end{proof}

	\section{Three levels of strong Compactness}\label{section:three:levels}
	
	This section presents the core component of the paper where the detailed  proof of Theorem~\ref{thm:main} is given.  For the sake of clarity, we are going to proceed in three steps, each of which is the subject of one of the subsequent sections.	
	
	The overall strategy of our proof consists, first, in establishing a stability result of the difference $\omega_\lambda-\Omega$ in a space with a negative regularity. Thereafter, in a combination with an abstract extrapolation compactness argument, the proof of Theorem \ref{thm:main} will be reduced  to the study of the evanescence of some specific high frequency modes in a higher regularity space. 
	\subsection{Convergence in low-regularity spaces}\label{section:level:one}

	We begin with a first convergence result of the vorticities $\omega_\lambda-\Omega$, solving \eqref{QGSW} and \eqref{Euler} respectively, as $\lambda \to 0$,   in a space that is slightly better than $B^{-1}_{2,\infty}$. More precisely, we now introduce the space  
	\begin{equation*}
		\norm f_{X^{s}_{p,r}} \bydef \norm { S_0\nabla ^\perp \Delta^{-1} f}_{L^p} + \left( \sum_{j\in \mathbb N} 2^{rjs} \norm {\Delta_j f}_{L^p}^r \right) ^\frac{1}{r},
	\end{equation*}
	for any $s\in \mathbb R$ and $p,r\in [1,\infty]$, with the usual adjustment for $r=\infty$.
	From the above setting, and due to the identity 
	\begin{equation*}
		f= -\curl \nabla^\perp \Delta^{-1} f,
	\end{equation*} it is readily seen that 
	\begin{equation}\label{Besov_X_bound}
		\norm {f}_{B^{s}_{p,r}} \lesssim \norm f_{X^{s}_{p,r}},
	\end{equation}
	for any values of the  parameters $s\in \mathbb R$ and $p,r\in [1,\infty]$.
	In practical situations, it is  useful to keep in mind the statement that
	\begin{equation*}
		f\in X^s_{p,r} \qquad   \text{if and only if} \qquad  S_0 \nabla ^\perp \Delta^{-1} f \in L^p \quad \text{and} \quad  (\id - S_0) f \in  B^s_{p,r},
	\end{equation*} 
with the equivalent norm 
\begin{equation*}
	\norm f_{X^s_{p,r}} \sim \norm {S_0 \nabla ^\perp \Delta^{-1} f }_{L^p} + \norm {(\id - S_0) f}_{B^s_{p,r}}.
\end{equation*}
	
	  The following theorem summarizes the central convergence result  of this section.
	\begin{thm}\label{thm:conv:low:reg}
		Given the setup above, and assume  that the initial data enjoys the bounds 
		\begin{equation*}
			 \omega_{0,\lambda}, \Omega \in L^1\cap  L^\infty (\mathbb R^2),
		\end{equation*} 
		uniformly in $\lambda \in (0,1)$. Assuming further that
		\begin{equation*}
			\lim_{\lambda\to 0}\norm {\omega_{0,\lambda} -\Omega_0}_{X^{-1}_{2,\infty}}=0,
		\end{equation*} 
		then, it holds that 
		\begin{equation*}
			\lim_{\lambda\to 0} \norm {\omega_\lambda-\Omega}_{L^\infty ([0,t]; X^{-1}_{2,\infty})} =0.
		\end{equation*} 
		More precisely, there is a constant $C_{**}>0$ depending only on the size of the initial data such that, for any given time $T>0$, it holds that 
		\begin{equation*}
			\norm {\omega_\lambda-\Omega }_{L^\infty([0,T];X^{-1}_{2,\infty})}  \leq  \left( \norm {\omega_{0,\lambda}-\Omega_0}_{X^{-1}_{2,\infty}}^2 + C_{**} \lambda T \right)  ^{\frac{1}{2} \exp (-C_{**}T)},
		\end{equation*}
		as soon as $\lambda $ is small enough in such a way that 
		\begin{equation*}
			\left( \norm {\omega_{0,\lambda}-\Omega_0}_{X^{-1}_{2,\infty}}^2 + C_{**} \lambda   T \right)  ^{\frac{1}{2} \exp (-C_{**}T)} \leq 1.
		\end{equation*}
	\end{thm}

	\begin{proof}
	We introduce the difference 
	\begin{equation*}
		f_\lambda\bydef \omega_\lambda-\Omega,
	\end{equation*}
	which is governed by the forced transport equation 
	\begin{equation}\label{f:equa}
		\partial_t f_\lambda  + v \cdot \nabla f_\lambda   =  \left( u_\lambda-v\right)  \cdot \nabla \omega_\lambda ,
	\end{equation}
	where the difference of the velocities can be recast as
	\begin{equation*}
		u_\lambda-v = \nabla^\perp (-\Delta ^{-1}) f_\lambda - \lambda \nabla^\perp  (-\Delta)^{-1} (\lambda - \Delta)^{-1}\omega_\lambda.
	\end{equation*}

		Let us now handle the low frequency part at first. To that end, applying the operator $\nabla^\perp \Delta^{-1}S_0$ to \eqref{f:equa}, and performing a standard energy estimate, yields that 
		\begin{equation*}
			\begin{aligned}
				\frac{1}{2} \frac{d}{dt} \norm {\nabla^\perp \Delta^{-1}S_0 f_\lambda}_{L^2}^2 
				&\leq   \norm {\nabla^\perp \Delta^{-1}\div  S_0  \left( vf_\lambda \right) }_{L^2}  \norm {\nabla^\perp \Delta^{-1}S_0 f_\lambda}_{L^2}
				\\
				& \quad +    \norm {\nabla^\perp \Delta^{-1}\div  S_0  \left( (u_\lambda-v)\omega_\lambda \right) }_{L^2}   \norm {\nabla^\perp \Delta^{-1}S_0 f_\lambda}_{L^2}
				\\
				&\lesssim  \norm {\nabla^\perp \Delta^{-1}\div  S_0  \left( vf _\lambda\right) }_{L^2}\norm {\nabla^\perp \Delta^{-1}S_0 f_\lambda}_{L^2}
				\\
				&\quad   + \norm {\nabla^\perp \Delta^{-1}\div  S_0  \left( \nabla^\perp \Delta^{-1} f _\lambda \omega_\lambda\right) }_{L^2}  \norm {\nabla^\perp \Delta^{-1}S_0 f_\lambda}_{L^2}
				\\
				&\quad + \norm {\nabla^\perp \Delta^{-1}\div  S_0  \left( \lambda \nabla^\perp(-\Delta)^{-1} (\lambda - \Delta)^{-1}\omega_\lambda \omega_\lambda\right) }_{L^2} \norm {\nabla^\perp \Delta^{-1}S_0 f_\lambda}_{L^2}.
			\end{aligned}
		\end{equation*}
		Thus, it follows that 
		\begin{equation*}
			\begin{aligned}
				\frac{1}{2} \frac{d}{dt} \norm {\nabla^\perp \Delta^{-1}S_0 f_\lambda}_{L^2}^2&\lesssim  \big(\norm {  S_0  \left( vf_\lambda \right) }_{L^2}  + \norm {  S_0  \left( \nabla^\perp \Delta^{-1} f  _\lambda\omega_\lambda\right) }_{L^2} 
				\\
				& \quad + \norm {  S_0  \left( \lambda \nabla^\perp (-\Delta)^{-1} (\lambda - \Delta)^{-1}\omega_\lambda \omega_\lambda\right) }_{L^2} \big) \norm {\nabla^\perp \Delta^{-1}S_0 f_\lambda}_{L^2}
				\\
				&\bydef  (I+II+III) \norm {\nabla^\perp \Delta^{-1}S_0 f_\lambda}_{L^2}.
			\end{aligned}
		\end{equation*}
		We now estimate each term separately. For $I$, we first split it as   
		\begin{equation*}
			\begin{aligned}
				I 
				&\leq \norm { S_0( v  S_0  f_\lambda ) }_{L^2} +\norm {  S_0  \left( v (\id - S_0)f _\lambda\right) }_{L^2}
				\\
				&\leq \norm {  v}_{L^\infty} \norm {  S_0  f_\lambda  }_{L^2} +\norm {  S_0  \left( v (\id - S_0)f_\lambda \right) }_{L^2}.
			\end{aligned}
		\end{equation*}
		We claim that  
		\begin{equation}\label{localized:product}
			\norm {  S_0  \left( v (\id - S_0)f_\lambda \right) }_{L^2} \lesssim   \norm {v}_{B^{1}_{\infty,\infty}} \norm {(\id - S_0) f_\lambda}_{B^{-1}_{2,\infty}},
		\end{equation}
		which, together with the bound
		\begin{equation*}
			\norm {S_0 f_\lambda}_{L^2}\lesssim \norm {S_0\nabla ^\perp \Delta ^{-1} f_\lambda}_{L^2},
		\end{equation*}
		yields that 
		\begin{equation*}
			I\lesssim  \norm {v}_{B^{1}_{\infty,\infty}} \norm {  f_\lambda}_{X^{-1}_{2,\infty}},
		\end{equation*}
		where we also used the embedding $B^1_{\infty,\infty} \hookrightarrow L^\infty.$
		
		Before we proceed any further, let us justify \eqref{localized:product}. To that end, it is important to make use of the frequency localization due to the operator $S_0$. This, combined with Bony's decomposition, yields that  
		\begin{equation*}
			\begin{aligned}
				\norm {  S_0  \left( v (\id - S_0)f_\lambda \right) }_{L^2} 
				&\lesssim \norm {  S_0  \left( T_ {(\id - S_0)f_\lambda}v \right) }_{L^2} + \norm {  S_0  \left( T_v (\id - S_0)f_\lambda \right) }_{B^{-1}_{2,\infty}}
				\\ & \quad +  \norm {  S_0  R  \big (v, (\id - S_0)f _\lambda \big) }_{B^{0}_{2,1}}
				\\
				&\lesssim \norm v _{B^1_{\infty,\infty}}\norm {(\id - S_0) f_\lambda}_{B^{-1}_{2,\infty}}.
			\end{aligned}
		\end{equation*}
		As for the term $II$, we proceed in a similar way, that is  by exploiting the localization effect of $S_0$ and splitting  the product according Bony's decomposition, to infer that 
		\begin{equation*}
			\begin{aligned}
				II & \leq  \norm {   \nabla^\perp \Delta^{-1}  S_0 f _\lambda \omega_\lambda  }_{L^2}  + \norm {  S_0\big(  \nabla^\perp \Delta^{-1} (\id -  S_0) f _\lambda \omega_\lambda \big) }_{L^2}  
				\\
				& \lesssim  \norm {\omega_\lambda}_{L^\infty}  \left( \norm {   \nabla^\perp \Delta^{-1}    S_0  f _\lambda  }_{L^2}  +  \norm {    (\id-  S_0)  f  _\lambda }_{B^{-1}_{2,\infty} } \right) 
				\\
				&\sim  \norm {\omega_\lambda}_{L^\infty} \norm {f_\lambda}_{X^{-1}_{2,\infty}}  .
			\end{aligned}
		\end{equation*}
		Finally, for the term $III$, we are going to pay a particular attention to its control in order to obtain a better decay rate in terms of the parameter $\lambda$. To that end, we first begin by writing  
		\begin{equation*}
			\begin{aligned}
				S_0  \left( \lambda (-\Delta)^{-1} (\lambda - \Delta)^{-1}\omega_\lambda \omega_\lambda\right)  
				& =\widetilde S_0 S_0  \left( \lambda (-\Delta)^{-1} (\lambda - \Delta)^{-1}\omega_\lambda \omega_\lambda\right)
				\\
				&= \widetilde S_0 \Big [ S_0, \lambda (-\Delta)^{-1} (\lambda - \Delta)^{-1}\omega_\lambda \Big  ]  \omega_\lambda
				\\
				& \quad  + \widetilde S_0 \left( \lambda (-\Delta)^{-1} (\lambda - \Delta)^{-1}\omega_\lambda S_0 \omega_\lambda \right),
			\end{aligned}		  
		\end{equation*}
		where $\widetilde S_0$ is the convolution with a smooth compactly supported Fourier multiplied taking the value one on the ball $B(0,2)$. The  the first term on the right-hand side can be controlled by virtue of  the classical commutator estimate (see \cite[Lemma 2.97]{bcd11}) to obtain that  
		\begin{equation*}
			\norm { \widetilde S_0 \Big [ S_0, \lambda \nabla^\perp (-\Delta)^{-1} (\lambda - \Delta)^{-1}\omega_\lambda \Big  ]  \omega_\lambda  }_{L^2}\lesssim \lambda \norm { \nabla^\perp(-\Delta)^{-1} (\lambda - \Delta)^{-1}\omega_\lambda}_{\dot H^1} \norm {\omega_\lambda}_{L^\infty}.
		\end{equation*}
		
		On the other hand,   we write,  by employing Bony's decomposition, again, that 
		\begin{equation*}
			\begin{aligned}
				\lambda\widetilde S_0 \left(  (-\Delta)^{-1} \nabla^\perp(\lambda - \Delta)^{-1}\omega_\lambda S_0 \omega_\lambda \right)
				&= \lambda \widetilde S_0 \left( T_{ \nabla^\perp(-\Delta)^{-1} (\lambda - \Delta)^{-1}\omega_\lambda}S_0 \omega_\lambda\right)
				\\
				&\quad +   \lambda \widetilde S_0 \left( T_{S_0 \omega_\lambda}{ \nabla^\perp(-\Delta)^{-1} (\lambda - \Delta)^{-1}\omega_\lambda} \right) 
				\\
				&\quad +\lambda \widetilde S_0   R\left( { \nabla^\perp(-\Delta)^{-1} (\lambda - \Delta)^{-1}\omega_\lambda},S_0 \omega_\lambda\right)
				\\
				& \bydef III_1+III_2 + III_3.
			\end{aligned}
		\end{equation*}
		First, we observe that  		\begin{equation*}
			III_1=0.
		\end{equation*}
		This, indeed, is a consequence of the  definition of low-high para-product component implies 
		\begin{equation*}
			\begin{aligned}
				T_a S_0 b 
				&= \sum_{j \in \mathbb Z} S_{j-1}a \Delta _jS_0b
				\\
				& =  \sum_{0\leq k\leq j-2} \sum_{-1\leq  j \leq 0} \Delta _{k-1}a \Delta _j \Delta_{-1} b,  
			\end{aligned} 
		\end{equation*}
		for any $a,b\in \mathcal S'(\mathbb R^2)$, whence the sum is identically zero, due to the restriction on the indices $(j,k)$ and the definition of the dyadic operators.

		Next, we estimate $III_2$ by a direct application of para-product laws (see \cite[Theorem 2.82]{bcd11}, for instance) to find that 
		\begin{equation*}
			\begin{aligned}
				\norm { III_2}_{L^2} 
				&\lesssim \lambda  \norm {\omega_\lambda}_{\dot B^{-1}_{\infty,2}} \norm { \nabla^\perp(-\Delta)^{-1} (\lambda - \Delta)^{-1}\omega_\lambda}_{\dot H^1}
				\\
				&\lesssim \lambda  \norm {\omega_\lambda}_{L^2} \norm {\nabla^\perp (-\Delta)^{-1} (\lambda - \Delta)^{-1}\omega_\lambda}_{\dot H^1}.
			\end{aligned}
		\end{equation*}
		Finally, the classical estimates of the remainder in Bony's decomposition (see \cite[Theorem 2.85]{bcd11}) leads to the control
		\begin{equation*}
			\begin{aligned}
				\norm { III_3}_{L^2}  \lesssim \lambda  \norm {\omega_\lambda}_{L^2 } \norm { \nabla^\perp(-\Delta)^{-1} (\lambda - \Delta)^{-1}\omega_\lambda}_{\dot H^1} .
			\end{aligned}
		\end{equation*}
		All in all, we have shown that 
		\begin{equation*}
			III \lesssim \lambda \norm {\omega_\lambda}_{L^1\cap L^2}  \norm {\nabla^\perp (-\Delta)^{-1} (\lambda - \Delta)^{-1}\omega_\lambda}_{\dot H^1},
		\end{equation*}
		which, by further employing Lemma \ref{lemma:error}, yields that  
		\begin{equation*}
			III\lesssim  \lambda^\frac{1}{2}   \norm {\omega_\lambda}_{L^1\cap L^2} ^2  .
		\end{equation*}
		
		In the end, by gathering the preceding bounds, we end up with  
		\begin{equation*}
			\frac{d}{dt} \norm {\nabla^\perp \Delta^{-1}S_0 f_\lambda}_{L^2} ^2 \lesssim _{C_*} \left(  \lambda^\frac{1}{2}  + \norm {f_\lambda}_{X^{-1}_{2,\infty}} \right) \norm {\nabla^\perp \Delta^{-1}S_0 f_\lambda}_{L^2},
		\end{equation*}
		where $C_*>0$ is a constant that  depends only on    the size of initial data,  which essentially comes from the following bounds on the solutions of \eqref{QGSW} and \eqref{Euler} 
		\begin{equation}\label{bound:vorticities}
			\norm {(\omega_\lambda,\Omega)}_{L^\infty (\mathbb R^+; L^1\cap L^\infty)}  \leq C_*.
		\end{equation}
		By an integration in time, we conclude that  
		\begin{equation}\label{f:low:frequencies}
			\begin{aligned}
				\norm {\nabla^\perp \Delta^{-1}S_0 f_\lambda(t)}_{L^2}^2 
				& \leq  	\norm {\nabla^\perp \Delta^{-1}S_0 f_{0,\lambda}}_{L^2}^2  + \lambda   t + C_*  \int_0^t \norm {     f _\lambda(\tau )  }_{X^{-1}_{2,\infty} } ^2 d\tau ,
			\end{aligned}
		\end{equation}
		which establishes the desired bound for the low frequency part.
		We turn our attention now to the control of the high frequencies. To that end, we first localize the frequencies of the transport equation \eqref{f:equa} and write that 
		\begin{equation*}
			\partial_t \Delta_j f_\lambda  + v \cdot \nabla \Delta_j f_\lambda =  \Delta_j \left( \nabla^\perp \Delta^{-1} f_\lambda - \lambda \nabla^\perp  \Delta^{-1}(\lambda-\Delta)^{-1} \omega_\lambda\right)  \cdot \nabla \omega_\lambda + [\Delta_j , v\cdot \nabla ] f_\lambda,
		\end{equation*} 
		for any $j\geq 0$. 
		It then follows that 
		\begin{equation}\label{localized:f}
			\begin{aligned}
				\frac{1}{2dt } \norm {\Delta_j f_\lambda }_{L^2}^2 
				&\leq \int _{\mathbb R^2 } \Delta_j \left( \nabla^\perp \Delta^{-1} f_\lambda \cdot \nabla \omega_\lambda \right)      \Delta_j f_\lambda dx   
				\\
				&\quad - \lambda\int _{\mathbb R^2 } \Delta_j \left( \nabla^\perp  \Delta^{-1}(\lambda-\Delta)^{-1} \omega_\lambda \cdot \nabla \omega_\lambda \right)   \Delta_j f_\lambda dx  
				\\
				& \quad +\int _{\mathbb R^2 }   [\Delta_j , v\cdot \nabla ] f_\lambda    \Delta_j f_\lambda dx 
				\\
				& \bydef K_j^1 +K_j^2 +K_j^3 ,
			\end{aligned}
		\end{equation}
		for any $j\geq 0$.
		Now, we take care of each summand separately. For the first one, we proceed employing H\"older inequalities to infer  that 
		\begin{equation*}
			\begin{aligned}
				K_j^1 
				& \lesssim \norm {\Delta_j \div (\nabla^\perp \Delta^{-1} f _\lambda \omega_\lambda  )}_{L^2}  \norm { \Delta_j f_\lambda}_{L^2}
				\\
				&\lesssim    2^j  \norm { \nabla^\perp \Delta^{-1}  f _\lambda \omega_\lambda  } _{L^2}  \norm { \Delta_j f_\lambda}_{L^2}
				\\
				&\lesssim    2^{2j}  \norm { \nabla^\perp \Delta^{-1}  f _\lambda } _{L^2}\norm {    \omega_\lambda  } _{L^ \infty}  \sup_{j\geq 0} \left( 2^{-j}  \norm { \Delta_j f_\lambda}_{L^2} \right)
				\\
				&\lesssim    2^{2j} \left(   \norm {S_0 \nabla^\perp \Delta^{-1}  f _\lambda } _{L^2} + \norm { (\id-S_0)   f _\lambda } _{\dot H^{-1}} \right) \norm {    \omega_\lambda  } _{L^ \infty}\norm {f_\lambda}_{X^{-1}_{2,\infty}}.
			\end{aligned}
		\end{equation*}
		Thus, by the interpolation inequality from Lemma \ref{lemma:interIneq}, we find that
		\begin{equation}\label{K1:bound}
			\begin{aligned}
				K_j^1 
				&\lesssim  2^{2j}   \norm {S_0 \nabla^\perp \Delta^{-1}  f _\lambda } _{L^2}   \norm {    \omega_\lambda  } _{L^ \infty}\norm {f_\lambda}_{X^{-1}_{2,\infty}}
				\\
				&\quad +  2^{2j}    \norm {    (\id-  S_0)  f  _\lambda }_{B^{-1}_{2,\infty} } \log \left( e+ \frac{ \norm {    (  \Omega, \omega_\lambda)   }_{L^2} }{ \norm {    (\id-  S_0)  f  _\lambda }_{B^{-1}_{2,\infty} }}\right)   \norm {    \omega_\lambda  } _{L^ \infty}\norm {f_\lambda}_{X^{-1}_{2,\infty}}.
			\end{aligned}
		\end{equation}
		As for the estimate of $K_j^2$,  employing similar arguments yields that 
		\begin{equation*}
			\begin{aligned}
				K_j^2
				& \lesssim \lambda  \norm {\Delta_j \left( \nabla^\perp  \Delta^{-1}(\lambda-\Delta)^{-1} \omega_\lambda \cdot \nabla \omega_\lambda \right)  }_{L^2}  \norm { \Delta_j f}_{L^2} 
				\\
				&   \lesssim 2^{2j}  \lambda  \norm { \left( \nabla^\perp  \Delta^{-1}(\lambda-\Delta)^{-1} \omega_\lambda   \omega_\lambda \right)  }_{L^2} \sup_{j\geq 0} \left( 2^{-j}  \norm { \Delta_j f}_{L^2} \right) 
				\\
				&   \lesssim 2^{2j}  \lambda  \norm { \left( \nabla^\perp  \Delta^{-1}(\lambda-\Delta)^{-1} \omega_\lambda   \right)  }_{L^\infty} \norm{ \omega_\lambda}_{L^2} \norm {f_\lambda}_{X^{-1}_{2,\infty}}.
			\end{aligned}
		\end{equation*}
		Therefore, by making use of Lemma \ref{lemma:error}, we find that 
		\begin{equation}\label{K2:bound}
			K_j^2 \lesssim  2^{2j} \lambda^\frac{1}{2}     \norm {\omega_\lambda}_{L^1\cap L^2}^2  \norm {f_\lambda}_{X^{-1}_{2,\infty}} .
		\end{equation}
		We finally turn our attention to the control of $K_j^3$, which hinges on    the commutator estimate from Lemma \ref{lemma:commutator}. More precisely, we first split that term as 
		\begin{equation*}
			\int _{\mathbb R^2 }   [\Delta_j , v\cdot \nabla ] f_\lambda    \Delta_j f_\lambda dx = \int _{\mathbb R^2 }  \left( \mathcal I_j + \mathcal J_j \right)  \Delta_j f_\lambda dx ,
		\end{equation*}
		where we follow the notations from  the decomposition laid out in Lemma  \ref{lemma:commutator}. Then, a basic application of H\"older inequality leads to 
		\begin{equation*}
			\begin{aligned}
				K_j^ 3 
				&\lesssim  2^{ 2j}  \sup_{j\geq 0} \left(  2^{-j} \norm {\mathcal I_j}_{L^{\frac{2q}{q+1}}} \right) \sup_{j\geq 0} \left( 2^{-j}  \norm { \Delta_j f_\lambda}_{L^ \frac{2q}{q-1}} \right)
				\\
				&\quad +   2^{ 2j}  \sup_{j\geq 0} \left(  2^{-j} \norm {\mathcal J_j}_{L^{2}} \right) \sup_{j\geq 0} \left( 2^{-j}  \norm { \Delta_j f_\lambda}_{L^2} \right),
			\end{aligned}
		\end{equation*} 
		where $q=q(t)\in (2,\infty)$ will be chosen later on.
		Therefore, we find, by employing the estimates from Lemma \ref{lemma:commutator}, that 
		\begin{equation*}
			K_j^3  
				 \lesssim  2^{ 2j} \norm {\nabla v}_{L^q}    \norm {  f_\lambda}_{B^{-1}_{ \frac{2q}{q-1},\infty}} ^2 + 2^{ 2j} \norm {\nabla v}_{ B^0_{\infty,\infty}}    \norm {  f_\lambda}_{B^{-1}_{2,1}} ^2.
		\end{equation*}
		By further employing Biot-Savart law 
		\begin{equation*}
			\norm {\nabla  v}_{L^q} \lesssim q \norm {\Omega}_{L^q} \lesssim q \norm {\Omega}_{L^2\cap L^\infty},
		\end{equation*}
		and the interpolation inequality 
		\begin{equation*}
			\begin{aligned}
				\norm { \Delta_j f_\lambda}_{L^ \frac{2q}{q-1}}^2 
				&\leq  \norm { \Delta_j f_\lambda}_{L^ 2}^{2-\frac{2}{q}} \norm { \Delta_j f_\lambda}_{L^ \infty} ^\frac{2}{q}
				\\
				&\lesssim \norm { \Delta_j f_\lambda}_{L^ 2}^{2-\frac{2}{q}} \left( 2^j \norm { \Delta_j f_\lambda}_{L^ 2} \right)^\frac{2}{q},
			\end{aligned}
		\end{equation*}
		we end up with the estimate 
		\begin{equation*} 
			K_j^ 3 
			\lesssim 2^{ 2j}  \norm {\Omega}_{L^2\cap L^\infty}\left(q \norm {  f}_{B^{-1}_{ 2,\infty}} ^{2-\frac{2}{q}}    \norm {  (\omega_\lambda,\Omega)}_{L^2}  ^{\frac{2}{q}} +  \norm {  f_\lambda}_{B^{-1}_{2,1}}\norm { (\id - S_0)  f_\lambda}_{B^{-1}_{2,\infty}} \right) .
		\end{equation*}
		Therefore, employing \eqref{Besov_X_bound} followed by Lemma \ref{lemma:interIneq}, we obtain that 
		\begin{equation}\label{K3:bound}
			\begin{aligned}
				K_j^ 3 
			&\lesssim 2^{ 2j}  \norm {\Omega}_{L^2\cap L^\infty} q \norm {  f}_{B^{-1}_{ 2,\infty}} ^{2-\frac{2}{q}}    \norm {  (\omega_\lambda,\Omega)}_{L^2}  ^{\frac{2}{q}} 
			\\
			& \quad + 2^{ 2j}  \norm {\Omega}_{L^2\cap L^\infty} \left( \norm { S_0 \nabla^\perp \Delta^{-1} f_\lambda}_{L^2}^2 + \norm {    (\id-  S_0)  f _\lambda  }_{B^{-1}_{2,\infty} } ^2\log \left( e+ \frac{ \norm {    (  \Omega, \omega_\lambda)   }_{ L^2}^2 }{ \norm {    (\id-  S_0)  f _\lambda  }_{B^{-1}_{2,\infty} }^2}\right) \right)  .
			\end{aligned}
		\end{equation}

		All in all, incorporating the estimates  \eqref{K1:bound}, \eqref{K2:bound} and \eqref{K3:bound}  in \eqref{localized:f} yields that  
		\begin{equation*}
			\begin{aligned}
			\sup_{j\geq 0}	\norm { \Delta_j  f _\lambda(t)}_{  L^2}  ^2 
				&\leq  \sup_{j\geq 0}	\norm { \Delta_j  f_{0,\lambda}}_{  L^2}  ^2    +\lambda  \norm {\omega_\lambda}_{L^1\cap L^2}^2 
				\\
				& \quad + (1+\norm { (  \Omega, \omega_\lambda) }_{L^\infty_t(L^2\cap L^\infty)})^2\int_0^t  q(\tau) \norm {     f _\lambda (\tau ) }_{X^{-1}_{2,\infty} }^{2-\frac{2}{q(\tau)}}  d\tau 
				\\
				& \quad + (1+\norm { (  \Omega, \omega_\lambda) }_{L^\infty_t(L^2\cap L^\infty)})^2\int_0^t  \norm {      f _\lambda(\tau )  }_{X^{-1}_{2,\infty} } ^2\log \left( e+ \frac{ \norm {    (  \Omega, \omega_\lambda)   }_{L^\infty_t L^2}^2 }{ \norm {     f _\lambda (\tau ) }_{X^{-1}_{2,\infty} }^2}\right)   d\tau .
			\end{aligned}
		\end{equation*}
		By further employing the bound \eqref{bound:vorticities} on the vorticities, we end up with
		\begin{equation*}
			\begin{aligned}
			\sup_{j\geq 0}	\norm { \Delta_j  f_\lambda (t)}_{  L^2}  ^2 
				&\leq  \sup_{j\geq 0}	\norm { \Delta_j  f_{0,\lambda}}_{  L^2}  ^2    + C_{**}\lambda + C_{**}\int_0^t  q(\tau) \norm {     f _\lambda (\tau ) }_{X^{-1}_{2,\infty} }^{2-\frac{2}{q(\tau)}}  d\tau 
				\\
				& \quad + C_{**}\int_0^t  \norm {      f _\lambda(\tau )  }_{X^{-1}_{2,\infty} } ^2\log \left( e+ \frac{ eC_{**} }{ \norm {     f _\lambda (\tau ) }_{X^{-1}_{2,\infty} }^2}\right)   d\tau ,
			\end{aligned}
		\end{equation*}
		where the constant $C_{**}>0$ depends only on the initial data, uniformly in $\lambda\in (0,1)$.
		
		Let now $T>0$ be fixed, and we assume that 
		\begin{equation*}
			\norm {      f _\lambda(t )  }_{X^{-1}_{2,\infty} } ^2 \leq C_{**}
		\end{equation*}
		for any $t\in [0,t_*]$, for some $t_*\in (0,T]$. We then make the choice of the integrability exponent 
		\begin{equation*}
			q(t) = 2 \log \left( \frac{e C_{**}}{ \norm {   f (t)}_{  X^{-1}_{2,\infty}}  ^2} \right) , \quad \text{for all } t\in [0,t_*),
		\end{equation*} 
		which leads to 
		\begin{equation*}
			\begin{aligned}
				\sup_{j\geq 0}	\norm { \Delta_j  f_\lambda (t)}_{  L^2}  ^2  
				&\leq  \sup_{j\geq 0}	\norm { \Delta_j  f _{0,\lambda}}_{  L^2}  ^2   
				+C_{**} \lambda + C_{**}\int_0^t  \norm {      f_\lambda (\tau )  }_{X^{-1}_{2,\infty} } ^2\log \left( e+ \frac{ eC_{**} }{ \norm {     f_\lambda  (\tau ) }_{X^{-1}_{2,\infty} }^2}\right)   d\tau   .
			\end{aligned}
		\end{equation*}
		To conclude, we combine the preceding control of the high frequencies with \eqref{f:low:frequencies} on the low frequency part to obtain the ultimate estimate 
		\begin{equation*}
			\norm {f_\lambda(t)}_{X^{-1}_{2,\infty}}^2  \leq \norm {f_{0,\lambda}}_{X^{-1}_{2,\infty}}^2 + C_{**} \lambda t +  C_{**}\int_0^t  \norm {      f_\lambda (\tau )  }_{X^{-1}_{2,\infty} } ^2\log \left( e+ \frac{ eC_{**}}{ \norm {      f_\lambda  (\tau ) }_{X^{-1}_{2,\infty} }^2}\right)   d\tau  ,
		\end{equation*}
		for any $t\in [0,t_*)$.
		At last, applying Osgood lemma \cite[Lemma 3.4]{bcd11} yields that 
		\begin{equation*}
			\norm {f_\lambda(t)}_{X^{-1}_{2,\infty}}^2  \leq  \left( \norm {f_{0,\lambda}}_{X^{-1}_{2,\infty}}^2 + C_{**} \lambda  t \right)  ^{\exp (-C_{**}t)},
		\end{equation*}
		for any $t\in [0,t^*)$. In particular, assuming that $\lambda$ is small enough in the sense that     
		\begin{equation*}
			\norm {f_{0,\lambda}}_{X^{-1}_{2,\infty}}^2 + C_{**} \lambda   T \leq C_{**} e^{-\exp (-C_{**} T)}  < C_{**}
		\end{equation*}
		yields, by classical continuity argument, the validity of the bound 
		\begin{equation*}
			\norm {f_\lambda}_{L^\infty([0,T];X^{-1}_{2,\infty})}  \leq  \left( \norm {f_{0,\lambda}}_{X^{-1}_{2,\infty}}^2 + C_{**} \lambda  T \right)  ^{\frac{1}{2} \exp (-C_{**}T)}
		\end{equation*}
		for the initially fixed time $T>0$, thereby concluding the proof of the theorem.
	\end{proof}
	
	\subsection{Reduction of the problem: Extrapolation Compactness}\label{section:level:two}
	
	Given the convergence of the vorticities in $L^\infty_t B^{-1}_{2,\infty}$ that we established in the preceding section, we can now reduce the problem of the convergence in $ L^\infty_t L^2$ to the study of the evanescence of some suitable very high frequencies of the quasi-geostrophic shallow-water vorticity. This is the content of the next proposition.
	\begin{prop}\label{proposition:reduction}
		Under the assumptions of Theorem \ref{thm:conv:low:reg}, it holds that 
		\begin{equation*}
			\lim_{\lambda\to 0} \norm {\omega_\lambda-\Omega}_{L^\infty ([0,T]; L^2)} =0,
		\end{equation*}
		if and only if
		\begin{equation*}
			\lim_{\lambda\to 0} \norm { \mathds{1}_{|D|\geq \Theta_\lambda } \omega_\lambda}_{L^\infty ([0,T]; L^2)} =0,
		\end{equation*}
		for some $\Theta_\lambda>0$ satisfying 
		\begin{equation*}
			\lim _{\lambda\to 0} \Theta_\lambda = \infty \quad \text{ and } \quad \lim_{\lambda \to 0} \Theta_\lambda \norm {\omega_\lambda-\Omega}_{L^\infty ([0,T]; X^{-1}_{2,\infty})}=0.
		\end{equation*}
		Moreover, for   such  $\Theta_\lambda$ and any smooth    compactly supported function $\chi $, it holds that 
		\begin{equation*}
			\lim_{\lambda\to 0} \norm {\chi( \Theta_\lambda^{-1} |D|) \left(  \omega_\lambda-\Omega \right)}_{L^\infty ([0,T]; L^2)}  =0.
		\end{equation*}
		If furthermore $\chi $ is supported away from zero, then it holds that 
		\begin{equation*}
			\lim_{\lambda\to 0} \norm {\chi( \Theta_\lambda^{-1} D)  \omega_\lambda }_{L^\infty ([0,T]; L^p)}=0,
		\end{equation*}
		for any $p\in [2,\infty)$. 
	\end{prop}
	\begin{rem}
		In the preceding statement, and only in the case $p=2$, the function $\chi$ can be replaced by  a rough Fourier multiplier of  the form
		$$\mathds 1_{\{  \Theta_\lambda a \leq  |D| \leq  \Theta_\lambda b\}}(\cdot),$$
		for any $0<a< b <\infty$.  
	\end{rem}
	
	\begin{proof} 
		The direct implication readily follows from the triangular inequality 
		\begin{equation*}
			\norm { \mathds{1}_{|D|\geq \Theta_\lambda } \omega_\lambda}_{L^\infty ([0,T]; L^2)} \leq \norm { \mathds{1}_{|D|\geq \Theta_\lambda } \Omega}_{L^\infty ([0,T]; L^2)} +   \norm {   \omega_\lambda-\Omega  }_{L^\infty ([0,T]; L^2)},
		\end{equation*}
		and the fact that $\Omega \in C([0,T];L^2)$ without any dependence on $\lambda$. Notice that the continuity with respect to the time variable is very important here. We refer to \cite[Section 1.5.2]{ah2} for a comprehensive  discussion on that matter.
		
		As for the indirect  implication, we begin by writing that 
		\begin{equation*}
			\begin{aligned}
				\norm {   \omega_\lambda-\Omega  }_{L^\infty ([0,T]; L^2)}
				& \leq  \norm { S_0(  \omega_\lambda-\Omega ) }_{L^\infty ([0,T]; L^2)} +   \norm {(\id - S_0)(  \omega_\lambda-\Omega ) }_{L^\infty ([0,T]; L^2)} 
				\\
				& \lesssim  \norm { S_0\nabla ^\perp \Delta^{-1} ( \omega_\lambda-\Omega ) }_{L^\infty ([0,T]; L^2)} + \norm {\mathds 1_{|D|\geq  \Theta_\lambda}   \Omega ) }_{L^\infty ([0,T]; L^2)}  
				\\
				& \quad + \norm {\mathds 1_{|D|\geq  \Theta_\lambda}  \omega_\lambda}_{L^\infty ([0,T]; L^2)}   + \norm {\mathds 1_{|D|<\Theta_\lambda} (\id - S_0)(  \omega_\lambda-\Omega ) }_{L^\infty ([0,T]; L^2)} .
			\end{aligned}
		\end{equation*}
		Now, noticing that 
		\begin{equation}\label{low:frequency:cv}
			\begin{aligned}
				\norm {\mathds 1_{|D|<\Theta_\lambda} (\id - S_0)(  \omega_\lambda-\Omega ) (t) }_{ L^2} ^2
				& = \sum_{j=0}^{\log_2(\Theta_\lambda)}  \norm {\Delta_j (\omega_\lambda-\Omega )}_{L^2}^2
				\\
				& \leq \left(\sum_{j=0}^{\log_2(\Theta_\lambda)} 2^{2j} \right) \norm {(\id - S_0)( \omega_\lambda - \Omega)}_{B^{-1}_{2,\infty}}^2
				\\
				& \lesssim \Theta_\lambda^2  \norm {(\id - S_0)( \omega_\lambda - \Omega)}_{B^{-1}_{2,\infty}}^2,
			\end{aligned}
		\end{equation}
		it   follows, by employing the convergence result from the preceding section in     $X^{-1}_{2,\infty} \hookrightarrow B^{-1}_{2,\infty}$, that 
		\begin{equation*}
			\lim_{\lambda \to 0} \norm {   \omega_\lambda-\Omega  }_{L^\infty ([0,T]; L^2)} \leq \lim_{\lambda \to 0}  \norm {\mathds 1_{|D|\geq  \Theta_\lambda}   \omega_\lambda ) }_{L^\infty ([0,T]; L^2)} ,
		\end{equation*}
		as soon as 
		\begin{equation*}
			\lim_{\lambda \to 0} \Theta_\lambda\norm {   \omega_\lambda-\Omega_\lambda  }_{L^\infty ([0,T]; X^{-1}_{2,\infty})}= 0.
		\end{equation*}
		Once again, it is important to emphasize  that, here, the time-continuity of the  vorticity  $\Omega$  has been used to guarantee that  
		\begin{equation*}
			\lim_{\lambda \to 0} \norm {\mathds 1_{|D|\geq  \Theta_\lambda}  \Omega   }_{L^\infty ([0,T]; L^2)}  =0.
		\end{equation*}
		This completes the proof of the first part of the proposition. 
		
		As for   the convergence of the localized vorticities by any compactly supported $\chi$, it follows by following exactly the same approach presented in \eqref{low:frequency:cv}, whence we omit the details of that. 
		
		Finally, to establish our last claim, that is the evanescence of $\omega_\lambda$ when it is  localized on a rescaled annulus, we first notice that it is sufficient to prove the statement for $p=2$. Indeed, the convergence in the remaining range of parameters follows by interpolation. 
		
		Now, for $p=2$, we write that 
		\begin{equation*}
			\begin{aligned}
				\norm {\chi( \Theta_\lambda^{-1} |D|)  \omega_\lambda }_{L^\infty ([0,T]; L^2)} \leq \norm {\chi( \Theta_\lambda^{-1} |D|) ( \omega_\lambda - \Omega) }_{L^\infty ([0,T]; L^2)} + \norm {\chi( \Theta_\lambda^{-1} |D|)  \Omega }_{L^\infty ([0,T]; L^2)}.
			\end{aligned}
		\end{equation*}
		Therefore, the first term on the right-hand side vanishes in the limit $\lambda\to 0$ by virtue of the result of the preceding step, while the evanescence of the second term   is a consequence of the additional assumption that $\chi $ is supported away from the origin, and of course by employing the fact that $\Omega\in C([0,T];L^2)$, once again. The proof of the proposition is now completed.
	\end{proof}

	\subsection{Evanescence of high-frequency modes in high-regularity spaces}\label{section:level:three}
	
	Given the reformulation of the problem suggested in Proposition \ref{proposition:reduction}, the proof of the Theorem \ref{thm:main} is now reduced to the study of the evanescence of the very high frequencies 
	\begin{equation*}
		\mathds 1_{|D|> \Theta_\lambda} \omega_\lambda,
	\end{equation*}
	for a choice of a suitable cutoff parameter $\Theta_\lambda>0$ that will be made later on in such a way that 
	\begin{equation*}
		\quad 1 \ll  \Theta_\lambda  \ll   \left( \norm {\omega_\lambda-\Omega}_{L^\infty ([0,T]; X^{-1}_{2,\infty})}\right)^{-1} , \quad \text{for }  \lambda \ll 1.
	\end{equation*}
	This last piece of our proof is recovered by the next theorem, which presents the core part of this section.
	\begin{thm}\label{thm:high}
		Let $\omega_{0,\lambda}$ and $\Omega_0$ be an initial data for \eqref{Euler} and \eqref{QGSW}, respectively, belonging to the Yudivch class, i.e., 
		\begin{equation*}
			\omega_{0,\lambda}, \Omega_0 \in L^1\cap L^\infty(\mathbb R^2),
		\end{equation*}
		uniformly in $\lambda>0$. Assume further that 
		\begin{equation*}
			\lim_{\lambda\to 0} \left(  \norm {\nabla^\perp \Delta^{-1}(  \omega_{0,\lambda}-\Omega_0)}_{ L^2(\mathbb R^2)}+  \norm {\omega_{0,\lambda}-\Omega_0}_{ L^2(\mathbb R^2)}\right) =0.
		\end{equation*}
		Then, the family of functions  
		\begin{equation*}
			  \left( \mathds 1_{|D|\geq \Theta_\lambda} \omega_\lambda  \right)_{\lambda \ll1}
		\end{equation*}
		is strongly compact in $L^\infty ([0,T];L^2(\mathbb R^2))$, converging to zero, for any finite time $T>0$ and any family of parameters $(\Theta_\lambda)_{\lambda\in (0,1)}$ satisfying 
		\begin{equation*}
			\lim _{\lambda\to 0} \Theta_\lambda = \infty \quad \text{ and } \quad \lim_{\lambda \to 0} \Theta_\lambda \left(  \norm {\nabla^\perp \Delta^{-1}(  \omega_{0,\lambda}-\Omega_0)}_{ L^2(\mathbb R^2)} ^2 + C_{**} \lambda  T \right)  ^{\frac{1}{2} \exp (-C_{**}T)}=0.
		\end{equation*}
			\end{thm}

	As we emphasized earlier in Section \ref{section:challenge_and_proof}, in our discussion below, we actually will     establish a slightly stronger statement that, eventually, allows us to deduce Theorem \ref{thm:high} as a corollary. We recall that the more general setting of the problem here is the following: Let $g$ be a (scalar) solution of the (non)linear transport equation
	\begin{equation}\label{transport:EQ}
		\partial_t g + h\cdot \nabla g = F, \qquad g|_{t=0}=g_0,
	\end{equation}
	where $F$ is a source term, and $h$ is a divergence-free vector field (both potentially depend on the unknown  $g$).  
 Assuming   that   $h$, $g_0$ and/or $F$, and eventually $g$, depend on some parameter $\delta\in (0,1)$ in such a way that the dependence of the   problem's inputs (i.e., $F$, $g_0$, and possibly $h$ in a rough sense) is compact in Lebesgue spaces.  we then would like to investigate whether this compactness property can be extended to $g(t)$, for any positive times $t>0$.
 
	We give in the statement of 	Proposition \ref{prop:high:transport} below a quantitative bound on the behavior of the very high frequencies of $g$, which is essential in deriving a complete statement on the compactness of $g(t)$,  for any positive times $t>0$, under reasonnable assumptions on the velocity field $h$, the source term $F$ and the initial datum $g_0.$

	Before we proceed any further, we introduce some notations. Following the analysis from  \cite{ah2}, for a given rescaling parameter $\Theta>0$, we define the  convolution operators  
	\begin{equation}\label{S0:def}
		\widetilde  S_0   \bydef \psi(\Theta ^{-1}D)    ,\qquad \widetilde \Delta_j \bydef \widetilde \varphi_j(  D) \bydef \varphi_j( \Theta^{-1} D), \quad j\in \mathbb Z,
	\end{equation}
	where $\psi $ and $\varphi$ refer to the building blocks of the Littlewood-Paley theory, previously introduced in Section \ref{section:littlewood_paley}. Note that, in this section, we will be using the homogeneous counterpart of the dyadic decomposition from Section \ref{section:littlewood_paley}. The difference compared to the inhomogeneous setting is that the identities 
	\begin{equation*}
		\varphi_j (D)=0, \quad \text{for all }  j\leq -2, \qquad \text{and} \qquad \varphi_{-1} = \psi ,
	\end{equation*}
 previously imposed in the inhomogeneous setting, 	are no longer being imposed now. Instead, the definition of $\varphi _j(D)$  is fully given by the rescaling identity \eqref{S0:def}, for any $j\in \mathbb Z$, and not just for $j\in \mathbb N$. Accordingly, one has the two-side partition of unity 
	$$ \widetilde  S_0+  \sum_{j\geq 0}\widetilde \Delta_j= \mathrm{Id}= \sum_{j\in \mathbb Z}\widetilde \Delta_j$$
	in the sense of $\mathcal S'(\mathbb R^2)$, excluding non-zero polynomials.  Further introducing the Fourier-multiplier operators
	$$ \sqrt{\widetilde S_0} \bydef \sqrt{\psi   (\Theta^{-1} D)}
	\qquad\text{and}\qquad
	\sqrt{\mathrm{Id}-\widetilde S_0}\bydef \sqrt{1-\psi  (\Theta^{-1}D)} ,$$
	and observing that
	\begin{equation*}
		\left( \sum_{j\geq 0} \widetilde \varphi_j(\xi) \right)^2
		\leq \sum_{j\geq 0} \widetilde \varphi_j(\xi), \quad \text{for all } \xi \in \mathbb R^2,
	\end{equation*}
	it holds, for any given $g\in L^2(\mathbb R^2)$, that
	\begin{equation*}
		\norm{ (\mathrm{Id}-\widetilde  S_0 )g }_{L^2(\mathbb R^2)}
		\leq
		\norm{ \sqrt{\mathrm{Id}-\widetilde  S_0 }g}_{L^2(\mathbb R^2)},
	\end{equation*}
	which will come in handy later on.
	
	We are now in a position to state the main result of this section, which is about the behavior of the high Fourier modes of the solution   of \eqref{transport:EQ}. This will be employed afterwards to prove Theorem \ref{thm:high}.

	\begin{prop}\label{prop:high:transport}
		Let  $g$ be a solution of \eqref{transport:EQ} enjoying the bound  
		\begin{equation*}
			g\in L^\infty ([0,T];L^2 \cap L^\frac{2m}{m-2} \cap L^\frac{2q}{q-1} (\mathbb R^2)),
		\end{equation*}
		for some $T>0$ and $q,m\in [2,\infty]$. Then, it holds that 
		\begin{equation*}
			\begin{aligned}
				\frac{1}{2 }\norm{  \sqrt{\mathrm{Id} - \widetilde S_0} g (t )   }_{L^2}^2 & \leq  \frac{1}{2 }\norm{  \sqrt{\mathrm{Id} - \widetilde S_0} g_0 }_{L^2}^2+ \int_0^t  \norm{  \sqrt{\mathrm{Id} - \widetilde S_0}F(\tau)}_{L^2} \norm{  \sqrt{\mathrm{Id} - \widetilde S_0}g(\tau)}_{L^2}d\tau 
				\\ & \quad + C\sum_{i=-1}^0 \int_0^t\norm {\nabla h(\tau)}_{L^q}  \norm {\widetilde\Delta_i g(\tau)}_{L^\frac{2q}{q-1} }^2 d\tau 
				\\
				&\quad +  C \int_0^t\norm { \mathds{1}_{|D|\geq  \frac{1}{12}\Theta }\nabla h (\tau) }_{L^m}  \norm {   g(\tau)}_{L^\frac{2m}{m-2}}   \norm {   \widetilde\Delta_0g(\tau)}_{L^2}   d\tau 
				\\
				&\quad +  C \int_0^t\norm { \mathds{1}_{|D|\geq  \frac{1}{12}\Theta }\nabla h (\tau) }_{L^m}  \norm {   g(\tau)}_{L^\frac{2m}{m-2}}  \norm {  \sqrt{ \mathrm{Id}-\widetilde S_0 }g(\tau)}_{L^2}   d\tau ,
			\end{aligned}
		\end{equation*}
		for some universal constant $C>0$, and any $t\in [0,T]$.
	\end{prop}
	
	\begin{rem} One possible implication of the bound from the preceding lemma is that 
	\begin{equation*}
			\begin{aligned}
				 e^{-Ct}\norm{  \sqrt{\mathrm{Id} - \widetilde S_0} g   }_{L^\infty([0,t]; L^2)}^2 & \lesssim  \norm{  \sqrt{\mathrm{Id} - \widetilde S_0} g_0 }_{L^2}^2+ \left(  \int_0^t \norm{  \sqrt{\mathrm{Id} - \widetilde S_0}F(\tau)}_{L^2} d\tau \right)^2
				\\ & \quad + \sum_{i=-1}^0 \int_0^t\norm {\nabla h(\tau)}_{L^q}  \norm {\widetilde\Delta_i g(\tau)}_{L^\frac{2q}{q-1} }^2 d\tau 
				\\
				&\quad +   \int_0^t\norm { \mathds{1}_{|D|\geq  \frac{1}{12}\Theta }\nabla h (\tau) }_{L^m}  \norm {   g(\tau)}_{L^\frac{2m}{m-2}} \norm {   \widetilde\Delta_0g(\tau)}_{L^2}    d\tau 
				\\
				&\quad +  \left( \int_0^t\norm { \mathds{1}_{|D|\geq  \frac{1}{12}\Theta }\nabla h (\tau) }_{L^m}  \norm {   g(\tau)}_{L^\frac{2m}{m-2}}   d\tau \right)^2.
			\end{aligned}
		\end{equation*}
		The takeaway from this is that the evolution of high Fourier modes of the solution $g$ of a transport equation, advected by a velocity field $h$ and forced by a source term $F$, is controlled by the same high modes of the inputs (that is the initial data, the velocity and the source term), modulo a localized  (lower-order) counterpart of the solution itself. In particular, compactness properties for the solution $g$ in $L^p$ spaces can be obtained as a consequences of   weaker compactness results on the solution $g$ itself.
	\end{rem}

	\begin{proof}

		Formally, taking the  inner product of  \eqref{transport:EQ}   with $  \sum_{j\geq 0}\widetilde \Delta_j g $ and using the divergence-free condition of $h$, we find that
		\begin{equation}\label{energy-HF}
			\begin{aligned}
				\frac{1}{2 }\norm{  \sqrt{\mathrm{Id} - \widetilde S_0} g (t )   }_{L^2}^2 &= \frac{1}{2 }\norm{  \sqrt{\mathrm{Id} - \widetilde S_0} g_0 }_{L^2}^2+ \int_0^t  \int_{\mathbb{R}^2}  F   (\mathrm{Id} - \widetilde S_0 ) g(\tau,x)  dxd\tau\\
				& \quad -\underbrace{ \int_0^t  \int_{\mathbb{R}^2}h\cdot \nabla \left( \sum_{j\leq -1}\widetilde\Delta_j g  \right)  \sum_{j\geq 0}\widetilde \Delta_j g(\tau,x)     dxd\tau }_{\bydef\mathcal{J}(t)} .
			\end{aligned}
		\end{equation}
		In fact, notice that the preceding bound cannot   be obtained by directly taking the inner product with $  \sum_{j\geq 0}\widetilde \Delta_j g $, as, in this case, one needs to justify the formal cancelation  
		\begin{equation*}
			\int_{\mathbb{R}^2} h(\tau,x)\cdot \nabla \left(\sum_{j\geq 0}\widetilde\Delta_jg(\tau,x) \right) \left(\sum_{j\geq 0}\widetilde\Delta_jg(\tau,x) \right)dx =0, \quad \text{for all } \tau \in [0,T].
		\end{equation*}
		Instead, as it is explained in \cite{ah2} for the Euler equations, one way to obtain \eqref{energy-HF} is by first writing the energy identity for the low Fourier modes
		\begin{equation*} 
			\begin{aligned}
				\frac 12 \norm {  \sqrt{\widetilde  S_0 }  g(t)  }_{L^2}^2  
				& = \frac 12 \norm {\sqrt{\widetilde S_0}g_0  }_{L^2}^2  + \int_0^t\int_{\mathbb{R}^2}   F     \widetilde  S_0g (\tau,x)   dxd\tau \\
				& \quad +\int_0^t  \int_{\mathbb{R}^2} h\cdot \nabla (\widetilde  S_0g )(\mathrm{Id}-\widetilde  S_0 ) g (\tau,x)  dxd\tau  ,
			\end{aligned}
		\end{equation*} 
		that   utilizing  the cancelation identity 
		\begin{equation*}
			\int_{\mathbb{R}^2} h(\tau,x)\cdot \nabla \left( \widetilde S_0 g(\tau,x) \right) \widetilde S_0 g(\tau,x)  dx =0, \quad \text{for all } \tau \in [0,T],
		\end{equation*}
		which holds due to the smoothing effect of the operator $\widetilde S_0$, for any $h(\tau,\cdot)\in L^2$ by density.
		Then, by further using the classical  identity
		\begin{equation*} 
			\begin{aligned}
				\frac 12\norm {   g(t)  }_{L^2}^2  
				& =  \frac 12 \norm {g_0  }_{L^2}^2  + \int_0^t\int_{\mathbb{R}^2}  Fg (\tau,x)    dxd\tau ,
			\end{aligned}
		\end{equation*}
		we see that \eqref{energy-HF} follows from the fact that 
		$$\norm{ g (t)}_{L^2}^2 = \norm{\sqrt{\widetilde  S_0 } g(t) }_{L^2}^2 + \norm{ \sqrt{\mathrm{Id}-\widetilde  S_0 }g(t)  }_{L^2}^2.$$ 
		
		Now, we get back to \eqref{energy-HF} where our main task is to estimate $\mathcal{J}(t)$. To that end, we split it into a sum of three terms
		\begin{equation*} 
			\begin{aligned}
				\mathcal{J}_1(t) &\bydef \int_0^t \int_{\mathbb{R}^2}      h  \cdot \nabla \left(\widetilde  \Delta_{-1}  g(\tau)  \right) \widetilde \Delta_0 g (\tau,x)   dx d\tau ,
				\\
				\mathcal{J}_2(t) &\bydef   \int_0^t  \int_{\mathbb{R}^2}     h   \cdot \nabla \left(  \sum_{j\leq -2}\widetilde \Delta_j g \right) \sum_{j\geq 0}\widetilde \Delta_jg(\tau,x)     dxd\tau ,
				\\
				\mathcal{J}_3(t) &\bydef   \int_0^t  \int_{\mathbb{R}^2}     h  \cdot \nabla \left(  \widetilde \Delta_{-1}g  \right) \sum_{j\geq 1}\widetilde \Delta_jg(\tau,x)   dxd\tau ,
			\end{aligned}
		\end{equation*}
		which we now control separately.
		The term $\mathcal{J}_1 $ is the most difficult to estimate. However, this has already been taken care of in \cite[Lemma 1.7]{ah2} where it has been shown that
		\begin{equation*} 
			|\mathcal{J}_1(t)| \lesssim \sum_{i=-1}^0 \int_0^t\norm {\nabla h(\tau)}_{L^q}  \norm {\widetilde\Delta_i g(\tau)}_{L^\frac{2q}{q-1} }^2 d\tau,
		\end{equation*} 
		for any $q\in [2,\infty]$. 
		As for $  \mathcal{J}_2$ and $\mathcal{J}_3$, we begin with exploiting the support localizations
		\begin{equation*}
			\begin{aligned}
				\supp \mathcal{F}\left(  \sum_{j\leq -2}\widetilde \Delta_jg  \right)
				&= \left\{ \xi\in \mathbb{R}^2: |\xi| \leq \frac{2\Theta }{3} \right\}, 
				\\
				\supp \mathcal{F}\left(  \sum_{j\geq 0}\widetilde\Delta_jg   \right)
				&= \left\{ \xi\in \mathbb{R}^2: |\xi| \geq  \frac{3\Theta}{4} \right\}  , 
			\end{aligned}
		\end{equation*}
		to deduce that  
		$$ \supp \mathcal{F}\left(  \sum_{j\leq -2} \widetilde \Delta_jg \sum_{j\geq 0}\widetilde \Delta_jg   \right) \subset  \left\{ \xi\in \mathbb{R}^2: |\xi| \geq  \frac{\Theta }{12} \right\}.$$
		Similarly,  the fact that
		\begin{equation*}
			\begin{aligned}
				\supp \mathcal{F}\left(   \widetilde\Delta_{-1}g  \right) &\subset \left\{ \xi\in \mathbb{R}^2: |\xi| \leq \frac{4\Theta }{3} \right\},
				\\
				\supp \mathcal{F}\left(  \sum_{j\geq 1}\widetilde\Delta_jg   \right)&\subset \left\{ \xi\in \mathbb{R}^2: |\xi| \geq  \frac{3\Theta }{2} \right\}  , 
			\end{aligned}
		\end{equation*}
		entails that
		$$  \supp\mathcal{F}\left(  \widetilde\Delta_{-1}g  \sum_{j\geq 1}\widetilde\Delta_jg   \right) \subset  \left\{ \xi\in \mathbb{R}^2: |\xi| \geq  \frac{ \Theta }{6} \right\}.$$
		Thus, we deduce that 
		$$\begin{aligned}
			\mathcal{J}_2(t) 
			&=   \int_0^t  \int_{\mathbb{R}^2}    \left( \mathds{1}_{|D|\geq  \frac{1}{12}\Theta }h \right)  \cdot \nabla \left(  \sum_{j\leq -2}\widetilde \Delta_jg  \right) \sum_{j\geq 0}\widetilde\Delta_jg(\tau,x)  dxd\tau ,
		\end{aligned}$$
		and 
		$$\begin{aligned}
			\mathcal{J}_3(t) 
			&=   \int_0^t  \int_{\mathbb{R}^2}    \left( \mathds{1}_{|D|\geq  \frac{1}{6}\Theta } h \right)  \cdot \nabla \left(  \widetilde \Delta_{-1}g  \right) \sum_{j\geq 1}\widetilde \Delta_jg(\tau,x)     dxd\tau.
		\end{aligned}$$
		
		Consequently,  as the velocity  $h$ in the above is localized in its high-frequency modes, we obtain, by a simple application of H\"older inequalities, for any $m\in [2,\infty]$, that 
		\begin{equation*} 
			\begin{aligned}
				|\mathcal{J}_2(t)| &\lesssim   \int_0^t\norm { \mathds{1}_{|D|\geq \frac{1}{12}\Theta }\nabla h(\tau) }_{L^m}  \norm {  ( \widetilde S_0 - \widetilde\Delta_{-1} )g(\tau)}_{L^{\frac{2m}{m-2}}}  \norm { (\mathrm{Id}-\widetilde S_0   )g(\tau)}_{L^2}  d\tau\\
				&\lesssim  \int_0^t\norm {\mathds{1}_{|D|\geq  \frac{1}{12}\Theta }\nabla h (\tau)  }_{L^m}  \norm {   g(\tau)}_{L^\frac{2m}{m-2}}  \norm { (\mathrm{Id}-\widetilde S_0   )g(\tau)}_{L^2}  d\tau ,
			\end{aligned}
		\end{equation*}  
		and, in a similar fashion, that
		\begin{equation*}
			\begin{aligned}
				|\mathcal{J}_3(t)| &\lesssim   \int_0^t\norm {\mathds{1}_{|D|\geq \frac{1}{6}\Theta }\nabla h(\tau)  }_{L^m}  \norm {   \widetilde \Delta_{-1}g(\tau)}_{L^\frac{2m}{m-2}}  \norm { (\mathrm{Id}-\widetilde S_0 - \widetilde \Delta_0 )g(\tau)}_{L^2}  d\tau\\
				&\lesssim  \int_0^t\norm { \mathds{1}_{|D|\geq  \frac{1}{12}\Theta }\nabla h (\tau) }_{L^m}  \norm {   g(\tau)}_{L^\frac{2m}{m-2}} \Big( \norm {   \widetilde\Delta_0g(\tau)}_{L^2}  + \norm { (\mathrm{Id}-\widetilde S_0 )g(\tau)}_{L^2} \Big) d\tau .
			\end{aligned}
		\end{equation*}  
		
		All in all, gathering the preceding bounds yields that 
		\begin{equation*}
			\begin{aligned}
				\mathcal J (t) 
				&\lesssim \sum_{i=-1}^0 \int_0^t\norm {\nabla h(\tau)}_{L^q}  \norm {\widetilde\Delta_i g(\tau)}_{L^\frac{2q}{q-1} }^2 d\tau 
				\\
				&\quad +   \int_0^t\norm { \mathds{1}_{|D|\geq  \frac{1}{12}\Theta }\nabla h (\tau) }_{L^m}  \norm {   g(\tau)}_{L^\frac{2m}{m-2}} \Big( \norm {   \widetilde\Delta_0g(\tau)}_{L^2}  + \norm { \sqrt{ \mathrm{Id}-\widetilde S_0 }g(\tau)}_{L^2} \Big) d\tau,
			\end{aligned}
		\end{equation*}
		which, once plugged into \eqref{energy-HF}, leads to the conclusion of the proof.
	\end{proof}

	We are now going to derive a proof of Theorem \ref{thm:high} from the control of the high Fourier modes of the solution of the general transport equation \eqref{transport:EQ}. 
	\begin{proof}[Proof of Theorem \ref{thm:high}] Applying Proposition \ref{prop:high:transport} for 
		$$g=\omega_\lambda, \qquad h= u_\lambda=\nabla ^\perp (\lambda - \Delta)^{-1} \omega_\lambda ,  \qquad F=0, \qquad \Theta= \Theta_\lambda ,$$
		and with the integrability parameters $q=m=2  $, 
		yields that 
		\begin{equation*}
			\begin{aligned}
			\norm{  \sqrt{\mathrm{Id} -  S_0^\lambda} \omega_\lambda (t )   }_{L^2}^2 & \lesssim \norm{  \sqrt{\mathrm{Id} -  S_0^\lambda} \omega_{0,\lambda} }_{L^2}^2 + \sum_{i=-1}^0 \int_0^t\norm {\nabla u_\lambda(\tau)}_{L^2}  \norm {\Delta_i^\lambda \omega_\lambda(\tau)}_{L^4 }^2 d\tau 
				\\
				&\quad +   \int_0^t\norm { \mathds{1}_{|D|\geq  \frac{1}{12}\Theta _\lambda}\nabla u_\lambda (\tau) }_{L^2}  \norm {  \omega_\lambda (\tau)}_{L^\infty}  \norm {   \Delta_0^\lambda\omega_\lambda (\tau)}_{L^2}      d\tau
				\\
				&\quad +   \int_0^t\norm { \mathds{1}_{|D|\geq  \frac{1}{12}\Theta _\lambda}\nabla u_\lambda (\tau) }_{L^2}  \norm {  \omega_\lambda (\tau)}_{L^\infty}  \norm {  \sqrt{ \mathrm{Id}- S_0^\lambda }\omega_\lambda(\tau)}_{L^2}  d\tau
				\\
				&\leq  \norm{  \sqrt{\mathrm{Id} -  S_0^\lambda} \omega_{0,\lambda} }_{L^2}^2
				 + t\norm {\nabla u_\lambda}_{L^\infty_t L^2}\norm {\omega_\lambda}_{L^\infty_{t,x}}   \sum_{i=-1}^0   \norm {\Delta_i^\lambda \omega_\lambda }_{L^\infty_t L^2 }  
				\\
				&\quad +    \norm {   \omega_\lambda}_{L^\infty_{t,x}}\int_0^t\norm { \mathds{1}_{|D|\geq  \frac{1}{12}\Theta_\lambda }\nabla u_\lambda (\tau) }_{L^2}   \norm {  \sqrt{ \mathrm{Id}- S_0^\lambda }\omega_\lambda(\tau)}_{L^2}  d\tau,
			\end{aligned}
		\end{equation*}
		where we employed the interpolation inequality 
		\begin{equation*}
			\norm {\Delta_i^\lambda \omega_\lambda(\tau)}_{L^4 }^2\lesssim \norm {  \omega_\lambda(\tau)}_{L^\infty } \norm {\Delta_i^\lambda \omega_\lambda(\tau)}_{L^2 }. 
		\end{equation*}
		Now, noticing that  
		$$\begin{aligned}
			\norm { \mathds{1}_{|D|\geq  \frac{1}{12}\Theta_\lambda }\nabla u_\lambda (\tau) }_{L^2}  ^2  &=  \norm { \mathds{1}_{\frac{1}{12} \Theta_\lambda \leq |D|< \frac{4	}{3} \Theta_\lambda}\nabla u _\lambda(\tau) }_{L^2}^2 + \norm {\mathds{1}_{|D|\geq \frac{4}{3}\Theta_\lambda}\nabla u_\lambda(\tau)  }_{L^2}^2  \\
			&\leq   \norm {\mathds{1}_{\frac{1}{12} \Theta_\lambda\leq |D|\leq \frac{8	}{3} \Theta_\lambda}\omega_\lambda(\tau)  }_{L^2}^2+  \norm { \sqrt{  \mathrm{Id}- S_0^\lambda}\omega_\lambda(\tau)  }_{L^2} ^2,
		\end{aligned}$$ 
		we arrive at the bound 
		\begin{equation*}
			\begin{aligned} \norm{  \sqrt{\mathrm{Id} -  S_0^\lambda} \omega_\lambda (t )   }_{L^2}^2 
				&\lesssim  \norm{  \sqrt{\mathrm{Id} -  S_0^\lambda} \omega_{0,\lambda} }_{L^2}^2 
				+  t\norm {\omega_\lambda}_{L^\infty_t (L^2\cap L^\infty)} ^2  \sum_{i=-1}^0   \norm {\Delta_i^\lambda \omega_\lambda }_{L^\infty_t L^2 }
				\\
				&\quad +    t\norm {   \omega_\lambda}_{L^\infty_{t,x}}\norm {\mathds{1}_{\frac{1}{12} \Theta_\lambda\leq |D|\leq \frac{8	}{3} \Theta_\lambda}\omega_\lambda   }_{L^\infty_t L^2} ^2
				\\
				&\quad +   \norm {   \omega_\lambda}_{L^\infty_{t,x}}\int_0^t    \norm {  \sqrt{ \mathrm{Id}- S_0^\lambda }\omega_\lambda(\tau)}_{L^2}^2  d\tau.
			\end{aligned}
		\end{equation*}
		Finally, applying Gr\"onwall's lemma yields that 
		\begin{equation*}
			\begin{aligned} e^{-Ct \norm {\omega_\lambda}_{L^\infty_{t,x}}}\norm{  \sqrt{\mathrm{Id} -  S_0^\lambda} \omega_\lambda (t )   }_{L^2}^2 
				&\lesssim \norm{  \sqrt{\mathrm{Id} -  S_0^\lambda} \omega_{0,\lambda} }_{L^2}^2 + t\norm {\omega_\lambda}_{L^\infty_t (L^2\cap L^\infty)} ^2  \sum_{i=-1}^0   \norm {\Delta_i^\lambda \omega_\lambda }_{L^\infty_t L^2 }
				\\
				&\quad +    t\norm {   \omega_\lambda}_{L^\infty_{t,x}}\norm {\mathds{1}_{\frac{1}{12} \Theta_\lambda\leq |D|\leq \frac{8	}{3} \Theta_\lambda}\omega_\lambda   }_{L^\infty_t L^2} ^2 .
			\end{aligned}
		\end{equation*}
		To conclude, observe that all the terms on the right-hand side vanish in the limit $\lambda\to 0$, either by strong compactness of the initial data and Proposition \ref{proposition:reduction}. Moreover, the condition 
		\begin{equation*}
			\lim_{\lambda \to 0} \Theta_\lambda \left(  \norm {\nabla^\perp \Delta^{-1}(  \omega_{0,\lambda}-\Omega_0)}_{ L^2(\mathbb R^2)} ^2 + C_{**} \lambda  T \right)  ^{\frac{1}{2} \exp (-C_{**}T)}=0
		\end{equation*}
		in the statement of the theorem is a weaker reformulation of that 
		\begin{equation*}
			\lim_{\lambda \to 0} \Theta_\lambda \norm {\omega_\lambda-\Omega}_{L^\infty ([0,T]; X^{-1}_{2,\infty})}=0
		\end{equation*}
		from Propostion \ref{proposition:reduction}, which is in turn a consequence of the  control 
		\begin{equation*}
			\norm {\omega_\lambda-\Omega }_{L^\infty([0,T];X^{-1}_{2,\infty})}  \leq  \left( \norm {\omega_{0,\lambda}-\Omega_0}_{X^{-1}_{2,\infty}}^2 + C_{**} \lambda T \right)  ^{\frac{1}{2} \exp (-C_{**}T)}
		\end{equation*}
		from Theorem \ref{thm:conv:low:reg}. The proof of the theorem is now  completed.		
	\end{proof}
	
	\section{Proof of the main theorems}
	
	 \subsection{Strong Convergence in Yudovich class}\label{section:proof:thm1}
	 Given the analysis from  the preceding sections, we are now ready to outline the details of  proof of Theorem \ref{thm:main}.
	\begin{proof}[Proof of Theorem \ref{thm:main}]
		The convergence in $L^\infty_tL^2$ is a direct consequence of  combining Proposition \ref{proposition:reduction} and Theorem \ref{thm:high}, whereas the convergence in $L^\infty_tL^p$, for any $p\in (1,\infty)$, follows by interpolation.
		
		As for the uniform convergence in the case of continuous initial vorticity $\Omega_0$, we proceed by first deducing a convergence of the flow maps, as a consequence of the strong convergence from the preceding step. To see that, we introduce the (backward) flow maps as the unique (log-Lipschitz) solutions of the ODEs
		\begin{equation*}
		 \Phi_\lambda (t,x) = x - \int_0^t  u_\lambda(\tau, \Phi_{\lambda}(\tau,x)) d\tau 
		\end{equation*}
		and 
		\begin{equation*}
			 \Phi (t,x) = x - \int_0^t  v(\tau, \Phi(\tau,x)) d\tau ,
		\end{equation*}
		for any $(t,x)\in \mathbb R^+\times \mathbb R^2$. It is then readily seen that 
		\begin{equation*}
			\begin{aligned}
				\norm {\Phi_\lambda(t,\cdot)-\Phi (t,\cdot)}_{L^\infty} 
				&\leq \int_0^t  \left|u_\lambda(\tau, \Phi_{\lambda}(\tau,x)) - v (\tau, \Phi_{\lambda}(\tau,x))\right|d\tau 
				\\
				&\quad + \int_0^t \left|  v (\tau, \Phi_{\lambda}(\tau,x))- v(\tau, \Phi(\tau,x)) \right| d\tau 
				\\
				& \leq  t \norm {u_\lambda - v}_{L^\infty([0,t];L^\infty)} 
				\\
				&\quad + \norm {v}_{L^\infty([0,t];LL)}\int_0^t \norm {(\Phi_\lambda -\Phi) (\tau ,\cdot)}_{L^\infty} \log\left( \frac{e}{\norm {(\Phi_\lambda -\Phi) (\tau,\cdot)}_{L^\infty} } \right)   d\tau .
			\end{aligned}
		\end{equation*}
		Hence, by the interpolation inequality (using \eqref{bound:vorticities}, as well)
		\begin{equation*}
			\begin{aligned}
				\norm {u_\lambda - v}_{L^\infty([0,t]:L^\infty)} 
				&\lesssim \norm {u_\lambda - v}_{L^\infty([0,t];B^0_{2,\infty})} ^\frac{1}{2} \norm {u_\lambda - v}_{L^\infty([0,t];B^1_{\infty,\infty})}  ^\frac{1}{2}
				\\
				&\lesssim \norm {\omega_\lambda - \Omega}_{L^\infty([0,t];B^{-1}_{2,\infty})} ^\frac{1}{2} \norm {\omega_\lambda - \Omega}_{L^\infty([0,t];L^\infty)}  ^\frac{1}{2}
				\\
				&\lesssim C_{*}^\frac{1}{2} \norm {\omega_\lambda - \Omega}_{L^\infty([0,t];B^{-1}_{2,\infty})} ^\frac{1}{2}  
			\end{aligned}
		\end{equation*}
		together with 
		\begin{equation*}
			\norm {v}_{L^\infty([0,t];LL)} \lesssim \norm {\Omega}_{L^\infty([0,t];L^1\cap L^\infty)} \lesssim C_*,
		\end{equation*}
		it follows, by virtue of Osgood lemma (see \cite[Lemma 3.4]{bcd11}), that 
		\begin{equation*}
			\norm {\Phi_\lambda(t,\cdot)-\Phi (t,\cdot)}_{L^\infty} \lesssim  \left( C_{*} t\norm {\omega_\lambda - \Omega}_{L^\infty([0,t];B^{-1}_{2,\infty})}    \right) ^{-\frac{1}{4} \exp (-C_*t)},
		\end{equation*}
		for any finite $t \geq 0.$ 
		In view of the convergence result from the preceding steps, this entails that 
		\begin{equation*}
			\lim_{\lambda\to 0}\norm {\Phi_\lambda -\Phi }_{L^\infty([0,t]:L^\infty)} =0.
		\end{equation*}
		In fact, this convergence can be further improved, by interpolation, to hold in some positive H\"older spaces (with time deteriorating regularity exponent, see \cite{C98}). This improvement will not be required in our analysis below and is therefore omitted here.
		
		With the convergence of the flow maps at hand, we recast the solutions of \eqref{QGSW} and \eqref{Euler}, respectively, as
		\begin{equation*}
			\omega_{\lambda} (t,x)= \omega_{0,\lambda}(\Phi_\lambda (t,x))  \qquad \text{ and } \qquad \Omega(t,x)= \Omega_{0}(\Phi (t,x)), 
		\end{equation*}
		for all $(t,x)\in \mathbb R^+\times \mathbb R^2$. It follows easily that 
		\begin{equation*}
			\lim_{\lambda\to 0}\norm {\omega_{0,\lambda}(\Phi _\lambda(t,\cdot)) - \Omega_{0}(\Phi (t,\cdot))}_{L^\infty}=0
		\end{equation*} 
		by combining  the uniform convergence of the flow maps and the assumption on the strong convergence of the initial data in $L^\infty(\mathbb R^2)$. This concludes the proof of the theorem.
	\end{proof}
	\subsection{Sharpness of the Uniform Convergence}\label{section:proof:thm2}
	In the proof of Theorem \ref{thm:main2} below, we will use a result from \cite{GPSY21} on the characterization of uniformly rotating/stationary solutions of some active scalar equations. Thus, for convenience, embark with a concise discussion about  the central notions of our interest in this section.
	
	 To begin with, the following assumption from \cite{GPSY21} (set here in the two-dimensional case) applies to   our analysis later on.
	
	\begin{conditionHD}\label{HK:cond}
		Let $\mathcal K\in C^1(\mathbb R^2\setminus \{0\})$ be radially symmetric with 
	\begin{equation*}
		\mathcal K'(r) >0, 
	\end{equation*}
	for any $r>0$, and 
	\begin{equation*}
		|\mathcal K'(r)| \lesssim \frac{1}{r^{\delta
		}}, \quad \text{for some } \delta\in [2,3),
	\end{equation*} 
	for any $r\in (0,1]$.
	\end{conditionHD}

	The next proposition provides some examples of kernels   that satisfy the preceding condition. It will be a cornerstone to apply Theorem \ref{thm:disc} below later on in the proof of Theorem \ref{thm:main2}. 
	
	\begin{prop}\label{prop:kernel:A}
		Let the dimension be equal to two. The  kernels associated to the operators 
		\begin{equation*}
			\Delta ^{-1}  , \qquad  - (\lambda - \Delta)^{-1},  \qquad  \Delta ^{-1}+ (\lambda - \Delta)^{-1} ,
		\end{equation*}
		i.e., 
		\begin{equation*}
			x  \mapsto  \frac{1}{2\pi } \log |x|  , \qquad x  \mapsto   - \frac{1}{2\pi } K_0(\sqrt \lambda |x|), \qquad x  \mapsto  \frac{1}{2\pi }\log |x|+  \frac{1}{2\pi }K_0(\sqrt \lambda |x|) ,
		\end{equation*}
		 enjoy Condition (HD)  for any $\lambda >0$, where $K_0$ is the modified $K$-Bessel function. 
	\end{prop} 
 
	\begin{proof}
	We only focus our attention on the proof of the monotonicity of the third kernel from the  list of operators mentioned above. On the other hand,  while justifying that  Laplace kernel satisfies Condition  (HD) is obvious,  the proof of that K-Bessel function can be deduced from the arguments below. We also refer to \cite[Lemma 2.2]{HZHH23} for a precise statement on the behavior of these kernels near the origin.
	
	Having fixed our objective, we now begin by noticing that the monotonicity of the function 
	\begin{equation*}
		x  \mapsto   K_0(\sqrt \lambda |x|) +   \log |x|
	\end{equation*}
	does not depend of the value of the positive parameter $\lambda$, due to the identity 
	\begin{equation*}
	   K_0(\sqrt \lambda r) +   \log r =    K_0(\sqrt \lambda r) +   \log  (\sqrt \lambda r ) - \log (\sqrt \lambda),
	\end{equation*}
	for any $\lambda , r>0$. Thus, henceforth, we restrict our analysis to the value $\lambda=1$. 
	
	Recasting the modified Bessel function in its integral representation (see \cite[Identities 9.6.23 and 9.6.27]{Abram}) allows us to write that 
	\begin{equation*}
		  K'_0(r)= -K_1(r)= - r  \int_1^\infty e^{-rt} (t^2-1)^{\frac{1}{2}} dt, \qquad \text{for all } r>0,
	\end{equation*}
	which can  equivalently be reformulated, by virtue of  the change of variable $t\mapsto s+1$, as 
	\begin{equation*}
		K'_0(r) = - r e^{-r } \int_0^\infty  e^{-rs} (s^2 + 2s)^{\frac{1}{2}} ds.
	\end{equation*}
	Accordingly, it follows that 
	\begin{equation*}
		K_0'(r) \geq - r e^{-r}  \int_0^\infty e^{-rs} (s+1)  dt =  -e^{-r} \left( 1+ \frac{1}{r}\right), \qquad \text{for all } r>0.
	\end{equation*}
	Therefore, we obtain that 
	\begin{equation*}
		K_0'(r) + \frac{1}{r} \geq   \frac{1-e^{-r}}{r} - e^{-r} >0, \qquad \text{for all } r>0,
	\end{equation*}
	where the last lower bound can be easily verified  by  means of an elementary calculus of variations. This completes the proof of the proposition. 
	\end{proof}

	Let us now state a simple version of a theorem from \cite{GPSY21}, which will come into play later on in the proof of Theorem \ref{thm:main2}.
	
	\begin{thm}\label{thm:disc}
		Let $D\subset \mathbb R^2$ be a simply connected domain with a rectifiable boundary, and let $\mathcal K$ be a kernel satisfying Condition (HD), above. If $D$ is a solution of the equation 
		\begin{equation*}
			\mathds 1_{D}* \mathcal K= const, \quad \text{on }\partial D,
		\end{equation*}
		then, up to a translation, $D$ is necessarily a disc. 
	\end{thm}

	\begin{rem}
		Note that  Theorem 4.2 from \cite{GPSY21} establishes  a much stronger statement. Nevertheless, for the sake of simplicity, we have recast it here in dimension two only, under a stronger assumption on the domain $D$, compared to the more general reformulation given in \cite{GPSY21}. The fact that the stronger assumption on $D$ in our statement above implies Condition (HD) from \cite{GPSY21} is discussed in  \cite[Section 4.2]{GPSY21}.
	
	Moreover, it is worth pointing out that   \cite[Theorem 4.2]{GPSY21} actually classifies the solution $D$ of the equation 
	\begin{equation*}
			\mathds 1_{D}* \mathcal K - \frac{\Omega}{2} |\cdot|^2 = const, \quad \text{on }\partial D,
		\end{equation*}
	for $\Omega\in (-\infty,0]$, to be a disc (up to a translation when $\Omega=0$).
	\end{rem}
	
	We conclude this preliminary discussion by stating a corollary whose proof straightforwardly follows by applying Theorem \ref{thm:disc} with the kernel  associated with the operator $\Delta^{-1} + (\lambda - \Delta)^{-1}$, all by taking into into account Proposition \ref{prop:kernel:A}, as well.
	\begin{cor}\label{corollary:disc}
		Let $D\subset \mathbb R^2$ be a simply connected domain with a rectifiable boundary, and let $\mathcal K_\lambda$ be the  kernel associated with the operator $\Delta^{-1} + (\lambda - \Delta)^{-1}$, for   $\lambda>0$.  If $D$ is a solution of the equation 
		\begin{equation*}
			\mathds 1_{D}* \mathcal K_\lambda= const, \quad \text{on }\partial D,
		\end{equation*}
		then, up to a translation, $D$ is necessarily a disc. 
	\end{cor}

	We are now in a position to prove Theorem \ref{thm:main2}.
	
	\begin{proof}[Proof of Theorem \ref{thm:main2}] Let $D$ be a nontrivial simply connected domain with a rectifiable boundary, and which is different than the disc. Consider the initial datum 
	\begin{equation*}
		\omega_{0,\lambda}= \Omega
		_0= \mathds 1_{D},
	\end{equation*}
	 and note that the assumption that $D$ differs from disc is necessary, otherwise the solutions $\omega_\lambda$ and $\Omega$, associated to the same initial datum above, would be identical to the   stationary solution $\mathds 1_{D}$ of both \eqref{QGSW} and \eqref{Euler}.
	 
	 Given this setup, the solutions $\omega_\lambda$ and $\Omega$ are explicitly given by 
	 \begin{equation*}
			\omega_{\lambda} (t,\cdot)= \mathds 1_{ \Phi_{\lambda} (t,D)}  \qquad \text{ and } \qquad \Omega (t,\cdot)=\mathds 1_{ \Phi  (t,D)} ,  
		\end{equation*}
		for any $t\geq 0$,   
		where    the flow maps $\Phi_{\lambda}$ and $\Phi$ are defined in  the proof of Theorem \ref{thm:main}, above. Before going any further, notice that 
		\begin{equation*}
			\text{Vol}(\Phi _\lambda (t,D))= \text{Vol}(\Phi  (t,D))= \text{Vol}(D) , 
		\end{equation*}
		for all $\lambda>0$ and $t\geq  0$, due to the incompressibility condition. Now, we fix $\lambda>0$ and $ 0\leq t_1 < t_2$, and   consider two potential scenarios: 
		  \begin{equation*}
				\text{Vol}\Big (\Phi _\lambda (t,D) \cap \Phi  (t,D)\Big ) =\text{Vol}(D),  \quad \text{for all } t \in (t_1,t_2),			\end{equation*} 
			or
			  \begin{equation*}
				  \text{Vol}\Big (\Phi _\lambda (t,D) \cap \Phi  (t,D)\Big )<\text{Vol}(D), \quad \text{for some } t\in (t_1,t_2) .
			\end{equation*} 
			We claim that the first possibility does not occur with our choice of initial datum, whereas the second one leads to the desired lower bound on the difference between the two solutions.
			
			 To that end, we observe that, in the case when 
			\begin{equation}\label{scenario:1}
				\text{Vol}\Big (\Phi _\lambda (t,D) \cap \Phi  (t,D)\Big ) =\text{Vol}(D) = \text{Vol}(\Phi _\lambda (t,D))= \text{Vol}(\Phi  (t,D)),
			\end{equation}
			for any $t\in (t_1,t_2)$, it follows that 
			\begin{equation*}
				 \Phi _\lambda (t,D) = \Phi  (t,D), \quad \text{a.e. in } \mathbb R^2, \quad \text{for all } t\in (t_1,t_2).
			\end{equation*}
			This implies that 
			\begin{equation*}
				\omega_\lambda (t,\cdot) = \Omega (t,\cdot) ,  \quad \text{a.e. in } \mathbb R^2, \quad \text{for all } t\in (t_1,t_2),
			\end{equation*}
			which  in turn leads, by subtracting   \eqref{Euler} from \eqref{QGSW}, to the equation 
			\begin{equation*}
				\nabla ^\perp \left(  \mathcal K_\lambda  * \Omega\right) \cdot \nabla \Omega=0,
			\end{equation*}
			where we recall that $ \mathcal K_\lambda$ is the kernel associated with the operator 
			\begin{equation*}
				  \Delta^{-1} +  (\lambda-\Delta)^{-1}  .
			\end{equation*}
			Recalling that $\Omega$ is a vortex patch entails that 
			\begin{equation*}
				\nabla ^\perp \left( \mathcal K_\lambda*  \mathds 1_{\Phi(t,D)} \right) \cdot \vec n_{t} =0, \quad \text{on } \partial \Phi(t,D), \quad \text{for all } t\in (t_1,t_2),
			\end{equation*}
			where $\vec n_t$ denotes the outward normal vector to the boundary $\partial \Phi(t,D) $, which in turn can be recast as 
			\begin{equation*}
				 \mathcal K_\lambda*  \mathds 1_{\Phi(t,D)} =const , \quad \text{on } \partial \Phi(t,D), \quad \text{for all } t\in (t_1,t_2).
			\end{equation*}
			Thus,  Corollary  \ref{corollary:disc} implies  that $   \Phi(t,D)$ is necessarily  a disc. This is already a contradiction as $D$, and whence $ \Phi(t,D)$, is different than the disc, by assumption. 
			
			This rules  out the possibility for scenario \eqref{scenario:1} to occur. 
				Alternatively, it must happen that 
				\begin{equation*}
					\text{Vol}\Big (\Phi _\lambda (t,D) \cap \Phi  (t,D)\Big )<\text{Vol}(D), \quad \text{for some } t\in (t_1,t_2),
				\end{equation*} 
				which entails the existence of some nontrivial  set $\Gamma_{t,\lambda}\subset \mathbb R^2$ with the properties that 
				\begin{equation*}
					\text{Vol}(\Gamma_{t,\lambda})>0, \qquad \Gamma_{t,\lambda} \subset \Phi_{t,\lambda} (t,D)\qquad \text{and} \qquad  \Gamma_{t,\lambda} \cap \Phi (t,D) =\emptyset .
				\end{equation*}
				Given this, it is then readily seen that 
				\begin{equation*}
					\begin{aligned}
						\norm {\omega_\lambda(t,\cdot)-\Omega(t,\cdot)}_{L^\infty (\mathbb R^2)}
						& \geq \sup_{x\in \Gamma_{t,\lambda}} |\mathds 1_{\Phi (t,D)}(x) - \mathds 1_{\Phi_\lambda (t,D)}(x)|=1,
					\end{aligned}
				\end{equation*}
				thereby completing the proof of the theorem.
	\end{proof}


	\bibliographystyle{plain} 
	\bibliography{plasma}

\end{document}